\documentclass[10pt, a4paper]{article}
\usepackage[a4paper, total={6in, 9in}]{geometry}

\pdfoutput=1

\usepackage{graphicx}
\usepackage{epsfig}
\usepackage{booktabs}
\usepackage{stmaryrd}
\usepackage{url}
\usepackage{longtable}
\usepackage[figuresright]{rotating}
\usepackage[ruled,lined,linesnumbered]{algorithm2e}
\usepackage{color}
\usepackage{trivfloat}
\usepackage{multirow}
\usepackage{subcaption}
\usepackage{amsmath,amstext,amsbsy,amsopn,amsthm,amscd,amsxtra,upref,amssymb}
\usepackage{bm}
\usepackage{grffile}

\usepackage[geometry]{ifsym}
\usepackage{tikz}
\usetikzlibrary{backgrounds,decorations.pathmorphing}
\usetikzlibrary{patterns}
\usetikzlibrary{external}
\usetikzlibrary{calc}
\newtheorem {example}    {Example}
\newtheorem {theorem}    {Theorem}
\newtheorem {lemma}      {Lemma}

\newtheorem {proposition}{Proposition}

\definecolor{lightgray}{rgb}{0.9,0.9,0.9}
\SetAlTitleFnt{textsc}
\SetAlCapFnt{\sc}




\trivfloat{zAlgorithmFloatTemp}
\makeatletter
\newenvironment{zAlgorithmFloat}[1][0.5cm]
	{\begin{zAlgorithmFloatTemp}\noindent\begin{lrbox}{\@tempboxa}\begin{minipage}{0.99\columnwidth}\begin{algorithm}[H]} 
	{\end{algorithm}\end{minipage}\end{lrbox}\colorbox{lightgray}{\usebox{\@tempboxa}}\end{zAlgorithmFloatTemp}}
\makeatother

\newcommand{\Real}{\mathbb{R}}
\newcommand{\Complex}{\mathbb{C}}

\newcommand{\qX}[1][k]{X^{\mathsf{adi}}_{#1}}
\newcommand{\qY}[1][k]{Y^{\mathsf{adi}}_{#1}}

\newcommand{\YY}{\tilde{Y}}
\newcommand{\cX}[1][k]{X^{\mathsf{cay}}_{#1}}

\newcommand{\cY}[1][k]{Y^{\mathsf{cay}}_{#1}}
\newcommand{\cZ}[1][k]{Z^{\mathsf{cay}}_{#1}}
\newcommand{\hX}[1][k]{X^{\mathsf{inv}}_{#1}}

\newcommand{\rX}[1][k]{X^{\mathsf{kry}}_{#1}}
\newcommand{\rY}[1][k]{Y^{\mathsf{kry}}_{#1}}

\newcommand{\lX}[1][k]{X^{\mathsf{lya}}_{#1}}
\newcommand{\LX}{X^{\mathsf{lya}}}
\newcommand{\lZ}[1][k]{Z^{\mathsf{lya}}_{#1}}
\newcommand{\lV}[1][k]{V^{\mathsf{lya}}_{#1}}
\newcommand{\lR}[1][k]{R^{\mathsf{lya}}_{#1}}

\newcommand{\hY}{\tilde{Y}}

\newcommand{\pXtil}{\tilde{X}^{\mathsf{proj}}}

\newcommand{\pHtil}{\tilde{\nice{H}}^{\mathsf{proj}}}

\newcommand{\conj}[1]{\overline{#1}}
\newcommand{\tr}{\ast}

\newcommand{\rank}{\operatorname{rank}}

\newcommand{\real}[1]{\operatorname{Re}\left(#1\right)}
\newcommand{\imag}[1]{\operatorname{Im}\left(#1\right)}

\newcommand{\ii}{\mathsf{i}}

\newcommand{\lam}{\lambda}
\newcommand{\Lam}{\Lambda}

\newcommand{\nice}[1]{\mathcal{#1}}

\newcommand{\spann}[1]{\operatorname{span}\{ #1 \}}

\newcommand{\LHP}{\mathbb{C}_-}
\newcommand{\RHP}{\mathbb{C}_+}

\newcommand{\argmin}{\operatorname{argmin}}

\newcommand{\mb}[1]{\left[\begin{array}{#1}}
\newcommand{\me}{\end{array}\right]}
\newcommand{\smb}{\left[\begin{smallmatrix}}
\newcommand{\sme}{\end{smallmatrix}\right]}

\newcommand{\change}[1]{#1}

\newcommand{\bench}[1]{\textsf{#1}}

\newcommand*\rot{\rotatebox[origin=c]{90}}

\begin{document}

\title{\sc{{RADI}: A low-rank ADI-type algorithm for large scale algebraic Riccati equations}}

\author{Peter Benner \thanks{Max Planck Institute for Dynamics of Complex Technical Systems, Magdeburg, Germany; e-mail: {benner@mpi-magdeburg.mpg.de}}
    \and Zvonimir Bujanovi\'{c} \thanks{University of Zagreb, Department of Mathematics, Zagreb, Croatia; e-mail: {zbujanov@math.hr}}
    \and Patrick K\"{u}rschner \thanks{Max Planck Institute for Dynamics of Complex Technical Systems, Magdeburg, Germany; e-mail: {kuerschner@mpi-magdeburg.mpg.de}}
    \and Jens Saak \thanks{Max Planck Institute for Dynamics of Complex Technical Systems, Magdeburg, Germany; e-mail: {saak@mpi-magdeburg.mpg.de}}}

\date{}

\maketitle

\begin{abstract}
    This paper introduces a new algorithm for solving large-scale con\-ti\-nu\-ous-time algebraic Riccati equations (CARE).
    The advantage of the new algorithm is in its immediate and efficient low-rank formulation, which is 
    a generalization of the Cholesky-factored variant of the Lyapunov ADI method.
    We discuss important implementation aspects of the algorithm, such as reducing the use of complex arithmetic and
    shift selection strategies.
    We show that there is a very tight relation between the new algorithm and three other algorithms for CARE 
    previously known in the literature---all of these seemingly different methods
    in fact produce exactly the same iterates when used with the same parameters: they are algorithmically different
    descriptions of the same approximation sequence to the Riccati solution.
    \\


\end{abstract}

%!TEX root = qADI_LR.tex

%%%%%%%%%%%%%%%%%%%%%%%%%%%%%%%%%%%%%%%%%%%%%%%%%%%%%%%% 
% .. Introduction
%%%%%%%%%%%%%%%%%%%%%%%%%%%%%%%%%%%%%%%%%%%%%%%%%%%%%%%%
\section{Introduction}
\label{intro}

The continuous-time algebraic Riccati equation,
\begin{equation}
    \label{intro:eq:ric}
    A^\tr X + X A + Q - X G X = 0, 
\end{equation}
where
$$ 
    Q = C^\tr C, \; G = B B^\tr, \quad A\in\Real^{n \times n},\; B\in\Real^{n \times m},\; C\in\Real^{p \times n},
$$
appears frequently in various aspects of control theory, such as linear-quadratic optimal regulator problems, $H_2$ and $H_\infty$ controller design
and balancing-related model reduction. While the equation may have many solutions, for such applications one is interested in finding a so-called 
stabilizing solution: the unique positive semidefinite solution $X\in\Complex^{n\times n}$ such that the matrix $A-GX$ is stable (i.e.~all of its eigenvalues belong to
$\LHP$, the left half of the complex plane). 
\change{If the pair $(A, G)$ is stabilizable (i.e.~$\rank[A-\lam I, \;G] = n$,
for all $\lam$ in the closed right half plane), and the pair $(A, Q)$ is detectable (i.e.~$(A^\tr, Q^\tr)$ is stabilizable),
then such a solution exists \cite{LanR95,Bin14}.}
These conditions are fulfilled generically, and we assume they hold throughout the paper.

There are several algorithms for finding the numerical solution of $\eqref{intro:eq:ric}$. In the case when $n$ is small enough, one can
compute the eigen- or Schur decomposition of the associated Hamiltonian matrix 
\begin{equation}
    \label{intro:eq:ham}
    \nice{H} = \mb{cc} A & G \\ Q & -A^\tr \me,
\end{equation}
and use an explicit formula for $X$, see \cite{Lau79,Bin14}.
However, if the dimensions of the involved matrices prohibit the computation of a full eigenvalue or Schur decomposition, specialized large-scale
algorithms have to be constructed. In such scenarios, Riccati equations arising in applications have additional properties:
$A$ is sparse, and $p, m \ll n$, thus making the matrices $Q$ and $G$ of very-low rank compared to $n$.
\change{In practice, this often implies} that the sought-after solution $X$ will have a low numerical rank \cite{BujB14}, and allows for construction
of iterative methods that approximate $X$ with a series of matrices stored in low-rank factored form.
Most of these methods are engineered as generalized versions of algorithms for solving a 
large-scale Lyapunov equation \cite{Sim13,BenS13}, which is a special case of \eqref{intro:eq:ric} with $G=0$.

The alternating directions implicit (ADI) method \cite{Wac88} is a well established iterative approach for computing solutions of Lyapunov and
other linear matrix equations.
There exists an array of ADI methods \cite{BenLP08,BenLT09,LevR93,LiW02,Pen00b,Sab07,Wac88,Wac00,Wac13}, covering both the ordinary and the generalized
case. 
All of these methods have simple statements and efficient implementations \cite{Pen00b}. % MESS citation: BMS08a. 
One particular advantage of ADI methods is that they are very well suited for large-scale problems: the default formulation
which works with full-size dense matrices can be transformed into a series of iteratively built approximations to the solution.  
Such approximations are represented in factored form, each factor having a very small rank compared to the dimensions of the input matrices.
This makes ADI methods very suitable for large-scale applications.

Recently, Wong and Balakrishnan \cite{WonB05,morWonB07} suggested a so-called quadratic ADI me\-thod (qADI) for solving the algebraic Riccati equation
\eqref{intro:eq:ric}.
Their method is a direct generalization of the Lyapunov ADI method, but only when considering the formulation working with full-size dense matrices.
However, in the setting of the qADI algorithm, it appears impossible to apply a so-called \change{``Li--White trick''} \cite{LiW02},
which is the usual method of obtaining a low-rank formulation of an ADI method.
Wong and Balakrishnan do provide a low-rank variant of their algorithm, but this variant has an important drawback: in each step, 
all the low-rank factors have to be rebuilt from scratch. This has a large negative impact on the performance of the algorithm.

Apart from the qADI method, there are several other methods for solving the large-scale Riccati equation that have appeared in the literature recently. 
Amodei and Buchot \cite{AmoB10}
derive an approximation of the solution by computing small-dimensional invariant subspaces of the associated Hamiltonian matrix \eqref{intro:eq:ham}.
Lin and Simoncini \cite{LinSim13} also consider the Hamiltonian matrix, and construct the solution by running subspace iterations on its Cayley transforms.
Massoudi, Opmeer, and Reis \cite{MasOR14} have shown that the latter method can be obtained from the control theory point of view as well.

\change{In this paper, we introduce a new ADI-type iteration for Riccati equations, RADI.
The derivation of RADI is not related to qADI, and it immediately gives the low-rank algorithm 
which overcomes the drawback from \cite{WonB05,morWonB07}.
The low-rank factors are built incrementally in the new algorithm: in each step, each factor is expanded by several columns and/or rows, 
while keeping the elements from the
previous steps intact. By setting the quadratic coefficient $B$ in \eqref{intro:eq:ric} to zero, our method reduces to the low-rank formulation
of the Lyapunov ADI method, see, e.g., \cite{morBenKS13,BenLP08,LiW02,Pen00b}.}

\change{A surprising result is that, despite their completely different derivations,
all of the Riccati methods we mentioned so far are equivalent: the approximations they produce
in each step are the same. This was already shown \cite{BujB14} for the qADI algorithm, and the
algorithm of Amodei and Buchot. In this paper we extend this equivalence to our new low-rank RADI method 
and the method of Lin and Simoncini. Among all these different formulations of the same
approximation sequence, RADI offers a compact and efficient implementation, and is very well suited for effective computation.}
% \textcolor{blue}{P{\"ad}: Most favorable is probably a bit too much. Or we should explain the meaning of most favorable.
% Zvonimir: Agreed, I've toned it down. Is this OK?}

This paper is organized as follows: in Section \ref{lrcf}, we recall the statements of the Lyapunov ADI method and the various Riccati methods,
and introduce the new low-rank RADI algorithm. 
\change{The equivalence of all aforementioned methods is shown in Section \ref{equiv}.
In Section \ref{impl} we discuss important implementation issues,
and in particular, various strategies for choosing shift parameters. 
Finally, Section \ref{num} compares the effect of different options for the algorithm on its performance 
via several numerical experiments.
We compare RADI with other algorithms for computing low-rank approximate solutions of \eqref{intro:eq:ric} 
as well: the extended \cite{HeyJ09} and rational Krylov subspace methods
\cite{SimSM13a}, and the low-rank Newton-Kleinman ADI iteration \cite{BenLP08,BenS10,BenHW15,Saa09}.
% , and algorithms based on doubling \cite{Bin14,ChuFL05}.
}

The following notation is used in this paper: $\LHP$ and $\RHP$ are the open left and right half plane, respectively,
while $\real{z},~\imag{z}$, $\conj{z}=\real{z}-\ii\imag{z}$, $|z|$
denote the real part, imaginary part, complex conjugate, and absolute value of a complex quantity $z$. 
For the matrix $A$, we use $A^\tr$ and $A^{-1}$ for
the complex conjugate transpose and the inverse, respectively. 
% \textcolor{blue}{P{\"ad}: Actually, we only use $A^\tr$. Zvonimir: OK, removed $A^T$.}
In most situations, expressions of the form $x=A^{-1}b$ are to be understood as solving the linear system $Ax=b$ of equations for $b$.
The relations $A>(\geq)0$, $A<(\leq)0$ stand for the matrix $A$ being positive or negative (semi)definite. 
Likewise, $A\geq(\leq) B$ refers to $A-B\geq(\leq)0$. 
If not stated otherwise, $\|\cdot\|$ is the Euclidean vector or subordinate matrix norm, and $\kappa(\cdot)$ is the associated condition
number.
%!TEX root = qADI_LR.tex

%%%%%%%%%%%%%%%%%%%%%%%%%%%%%%%%%%%%%%%%%%%%%%%%%%%%%%%% 
% .. A new low-rank factored iteration
%%%%%%%%%%%%%%%%%%%%%%%%%%%%%%%%%%%%%%%%%%%%%%%%%%%%%%%%
\section{A new low-rank factored iteration}
\label{lrcf}

We start this section by stating various methods for solving Lyapunov and Riccati equations, which will be used throughout the paper.
First, consider the Lyapunov equation
\begin{equation}
	\label{lrcf:eq:lyap}
	A^\tr \LX + \LX A + Q = 0, \quad Q = C^\tr C, \quad A\in\Real^{n \times n},\; C\in\Real^{p \times n}.
\end{equation}
Here we assume that $n$ is much larger than $p$.
When $A$ is a stable matrix, the solution $\LX$ is positive semidefinite. The Lyapunov ADI algorithm \cite{Wac13} generates
a sequence of approximations $(\lX)_k$ to $\LX$ defined by
\begin{equation}	
	\label{lrcf:eq:lyap-adi}
	\left.
		\begin{array}{rcl}
			\lX[k+1/2] (A+\conj{\sigma_{k+1}} I) &=& -Q - (A^\tr - \conj{\sigma_{k+1}} I)\lX[k], \\
			(A^\tr + \sigma_{k+1} I) \lX[k+1] &=& -Q - \lX[k+1/2] (A - \sigma_{k+1} I).
		\end{array}
	\right\}
	\text{ Lyapunov ADI}
\end{equation}
We will assume that the initial iterate $\lX[0]$ is the zero matrix, although it may be set arbitrarily.
The complex numbers $\sigma_k \in \LHP$ are called shifts, and the performance of ADI algorithms
depends strongly on the appropriate selection of these parameters \cite{Sab07}; this is further discussed in the context of the RADI method
in Section~\ref{ssne}.
Since each iteration matrix $\lX$ is of order $n$, formulation \eqref{lrcf:eq:lyap-adi} is unsuitable for large values of $n$, due to the
amount of memory needed for storing $\lX \in \Complex^{n \times n}$ and to the computational time needed for solving $n$ linear systems in each half-step.
The equivalent low-rank algorithm \cite{morBenKS13,BenKS14} generates the same sequence, but represents $\lX$ in factored form $\lX = \lZ (\lZ)^\tr$
with $\lZ \in \Complex^{n \times pk}$:
\begin{equation}	
	\label{lrcf:eq:lyap-adi-lr}
	\left.
		\begin{array}{rcl}
			\lR[0] &=& C^\tr, \\
			\lV[k] &=& \sqrt{-2\real{\sigma_k}} \cdot (A^\tr + \sigma_k I)^{-1}\lR[k-1], \\
			\lR[k] &=& \lR[k-1] + \sqrt{-2\real{\sigma_k}} \cdot \lV[k], \\
			\lZ[k] &=& \mb{cc} \lZ[k-1] & \lV[k] \me.
		\end{array}
	\right\}
	\text{ low-rank Lyapunov ADI}
\end{equation}

Initially, the matrix $\lZ[0]$ is empty.
This formulation is far more efficient for large values of $n$, since the right-hand side of the linear system in each step involves only the
tall-and-skinny matrix $\lR[k-1] \in \Complex^{n \times p}$.

We now turn our attention to the Riccati equation \eqref{intro:eq:ric}. Once again we assume that $p, m \ll n$, and seek to approximate
the stabilizing solution $X$. 
Wong and Balakrishnan \cite{WonB05,morWonB07} suggest the following quadratic ADI iteration (abbreviated as qADI):
\begin{equation}	
	\label{lrcf:eq:qadi}
	\left.
		\begin{array}{rcl}
			\qX[k+1/2] (A+\conj{\sigma_{k+1}} I - G \qX[k]) &=& -Q - (A^\tr - \conj{\sigma_{k+1}} I)\qX[k], \\
			(A^\tr + \sigma_{k+1} I - \qX[k+1/2] G) \qX[k+1] &=& -Q - \qX[k+1/2] (A - \sigma_{k+1} I).
		\end{array}
	\right\}
	\text{ quadratic ADI}
\end{equation}
Again, the initial approximation is usually set to $\qX[0]=0$. Note that by inserting $G=0$ in the quadratic iteration we obtain
the Lyapunov ADI algorithm \eqref{lrcf:eq:lyap-adi}.
As mentioned in the introduction, we will develop a low-rank variant of this algorithm such that inserting $G=0$ will reduce it
precisely to \eqref{lrcf:eq:lyap-adi-lr}.
To prepare the terrain, we need to introduce two more methods for solving the Riccati equation.

In the small scale setting, the Riccati equation \eqref{intro:eq:ric} is usually solved by computing the stable invariant subspace of the
associated $2n \times 2n$ Hamiltonian matrix \eqref{intro:eq:ham}.
To be more precise, let $1 \leq k \leq n$, and
\begin{equation}
	\label{lrcf:eq:ham-inv-subspace}
	\nice{H} \mb{c} P_k \\ Q_k \me = \mb{cc} A & G \\ Q & -A^\tr \me  \mb{c} P_k \\ Q_k \me =  \mb{c} P_k \\ Q_k \me \Lam_k,
\end{equation}
where $P_k, Q_k \in \Complex^{n \times k}$ and the matrix $\Lam_k \in \Complex^{k \times k}$ is stable.
For $k=n$, the stabilizing solution of \eqref{intro:eq:ric} is given by $X = -Q_k P_k^{-1}$.
In the large scale setting, it is computationally too expensive to compute the entire $n$-dimensional stable invariant subspace of $\nice{H}$.
Thus an alternative approach was suggested in \cite{AmoB10}: for $k \ll n$, one can compute only a $k$-dimensional, stable, invariant subspace and use an
approximation
given by the formula
\begin{equation}	
	\label{lrcf:eq:ham-ric}
	\left.
	\hX = -Q_k (Q_k^\tr P_k)^{-1} Q_k^\tr.
	\quad
	\right\}
	\text{ invariant subspace approach}
\end{equation}
Clearly, $\hX[n] = X$. This approach was further studied and justified in \cite{BujB14}, where it was shown that
$$
	\hX = \qX, \text{ for all } k,
$$
if $p=1$ and the shifts used for the {qADI} iteration coincide with the Hamiltonian eigenvalues of the matrix $\Lam_k$.
In fact, properties of the approximate solution $\hX$ given in \cite{BujB14} have lead us to the definition of the 
low-rank variant of the {qADI} iteration that
is described in this paper.
\\

The final method we consider also uses the Hamiltonian matrix $\nice{H}$.
The Cayley transformed Hamiltonian subspace iteration introduced in \cite{LinSim13}
generates a sequence of approximations $(\cX)_k$ for the stabilizing solution of the Riccati equation defined by
\begin{equation}	
	\label{lrcf:eq:cay}
	\left.
		\begin{array}{rcl}
			\mb{c} M_k \\ N_k \me &=& (\nice{H} - \sigma_k I)^{-1} (\nice{H} + \conj{\sigma_k}I) \mb{c} I \\ -\cX[k-1] \me,\\
			\cX &=& -N_k M_k^{-1},
		\end{array}
	\right\}
	\text{ Cayley subspace iteration}
\end{equation}
Here $\cX[0]$ is some initial approximation and $\sigma_k \in \LHP$ are any chosen shifts (Here we have adapted the notation to fit the one of this paper).
In Section~\ref{equiv} we will show that this method is also equivalent to the {qADI} and the new low-rank RADI iterations.

\subsection{Derivation of the algorithm}
The common way of converting ADI iterations into their low-rank variants is to perform a procedure similar to the one originally done by Li and White 
\cite{LiW02} for the Lyapunov ADI method. A crucial assumption for this procedure to succeed is that the matrices participating in the linear systems in
each of the half-steps mutually commute for all $k$. 
For the Lyapunov ADI method this obviously holds true for the matrices $A^\tr + \sigma_{k+1} I$.
However, in the case of the quadratic ADI iteration 
\eqref{lrcf:eq:qadi}, the matrices $A^\tr + \sigma_{k+1} I - \qX[k] G$ do not commute in general, for all $k$, 
and neither do $A^\tr + \sigma_{k+1} I - \qX[k+1/2] G$.

Thus we take a different approach in constructing the low-rank version. 
A common way to measure the quality of the matrix $\Xi\in\Complex^{n \times n}$ as an approximation to the Riccati solution
is to compute the norm of its residual matrix
$$
    \nice{R}(\Xi) = A^\tr \Xi + \Xi A + Q - \Xi G \Xi.
$$
The idea for our method is to repetitively update the approximation $\Xi$ by forming a so-called residual equation, until its solution converges to zero.
The background is given in the following simple result.
\begin{theorem}
	\label{lrcf:tm:residual_equation}
	Let $\Xi\in\Complex^{n\times n}$ be an approximation to a solution of \eqref{intro:eq:ric}.
	\begin{itemize}
		\item[(a)]
			Let $X=\Xi + \tilde{X}$ be an exact solution of \eqref{intro:eq:ric}. Then $\tilde{X}$ is a solution to the residual equation
			\begin{equation}
				\label{lrcf:eq:residual_equation}
				\tilde{A}^\tr \tilde{X} + \tilde{X} \tilde{A} + \tilde{Q} - \tilde{X}G\tilde{X} = 0,			
			\end{equation}
			where $\tilde{A} = A - G \Xi$ and $\tilde{Q} = \nice{R}(\Xi)$.
		\item[(b)]
			Conversely, if $\tilde{X}$ is a solution to \eqref{lrcf:eq:residual_equation}, then $X = \Xi + \tilde{X}$ is a solution
			to the original Riccati equation \eqref{intro:eq:ric}. 
			Moreover, if $\Xi \geq 0 $ and $\tilde{X}$ is a stabilizing solution to \eqref{lrcf:eq:residual_equation},
			then $X = \Xi + \tilde{X}$ is the stabilizing solution to \eqref{intro:eq:ric}.
		\item[(c)]
			If $\Xi \geq 0$ and $\nice{R}(\Xi) \geq 0$, then the residual equation \eqref{lrcf:eq:residual_equation} has a unique stabilizing
solution.
		\item[(d)]
			If $\Xi \geq 0$ and $\nice{R}(\Xi) \geq 0$, then $\Xi \leq X$, where $X$ is the stabilizing solution of \eqref{intro:eq:ric}.
\end{itemize}

% \newpage
\end{theorem}
\begin{proof} \ 
	\begin{itemize}
		\item[(a)]
			This is a straightforward computation which follows by inserting $\tilde{X} = X-\Xi$ and the formula for the residual of $\Xi$ into
\eqref{lrcf:eq:residual_equation}, see also \cite{Ben15,MehT88}.
		\item[(b)]
			The first part follows as in (a).
			If $\Xi \geq 0$ and $\tilde{X}$ is a stabilizing solution to \eqref{lrcf:eq:residual_equation},
			then $X = \Xi + \tilde{X} \geq 0$ and $A-GX = \tilde{A} - G\tilde{X}$ is stable, which makes $X$ the stabilizing solution to
\eqref{intro:eq:ric}.
		\item[(c), (d)]
			\change{The claims follow directly from (a) and \cite[Theorem 9.1.1]{LanR95}.}
	\end{itemize}
\end{proof}
Our algorithm will have the following form:
\begin{enumerate}
    \item Let $\Xi=0$. 
    \item Form the residual equation \eqref{lrcf:eq:residual_equation} for the approximation $\Xi$.
    \item Compute an approximation $\tilde{X}_1 \approx \tilde{X}$, where $\tilde{X}$ is the stabilizing solution of \eqref{lrcf:eq:residual_equation}.
    \item Accumulate $\Xi \leftarrow \Xi + \tilde{X}_1$, and go to Step 2.
\end{enumerate}

To complete the derivation, we need to specify Step 3 in a way that $\tilde{X}_1 \geq 0$ and $\nice{R}(\Xi + \tilde{X}_1) \geq 0$.
With these two conditions imposed, Theorem \ref{lrcf:tm:residual_equation} ensures that 
the residual equation in Step 2 always has a unique stabilizing solution and that the approximation $\Xi$ is kept positive semidefinite and 
monotonically increasing towards the stabilizing solution of \eqref{intro:eq:ric}.
The matrix $\tilde{X}_1$ fulfilling these conditions can be obtained by computing a 1-dimensional invariant subspace for the Hamiltonian matrix
associated with the residual equation, and plugging it into formula \eqref{lrcf:eq:ham-ric}.

More precisely, assume that 
$\nice{R}(\Xi) = \tilde{C}^\tr \tilde{C} \geq 0$, $G = BB^\tr$, and that $r, q \in \Complex^n$ satisfy
$$
	\mb{cc} 
		\tilde{A} & BB^\tr \\ 
		\tilde{C}^\tr \tilde{C} & -\tilde{A}^\tr 
	\me
	\mb{c} r \\ q \me
	= 
	\lam
	\mb{c} r \\ q \me,
$$
where $\lam \in \LHP$ is such that $-\lam$ is not an eigenvalue of $\tilde{A}$.
Equivalently,
\begin{align*}
	\tilde{A} r + BB^\tr q &= \lam r, \\
	\tilde{C}^\tr \tilde{C} r - \tilde{A}^\tr q &= \lam q.
\end{align*}

From the second equation we get $q = (\tilde{A}^\tr + \lam I)^{-1} \tilde{C}^\tr (\tilde{C}r)$. 
Let $$
\tilde{V}_1 = \sqrt{-2\real{\lam}}(\tilde{A}^\tr + \lam I)^{-1} \tilde{C}^\tr.
$$
Multiply the first equation by $q^\tr$ from the left, and the transpose of the second by $r$ from the right; then add the terms to obtain
$$
	q^\tr r = \frac{1}{2\real{\lam}} (q^\tr BB^\tr q + r^\tr \tilde{C}^\tr \tilde{C} r)
		= \frac{1}{2\real{\lam}} (\tilde{C}r)^\tr \left( I - \frac{1}{2\real{\lam}} (\tilde{V}_1^\tr B)(\tilde{V}_1^\tr B)^\tr \right)
(\tilde{C}r).
$$
Expression \eqref{lrcf:eq:ham-ric} has the form 
$\tilde{X}_1 = -q (q^\tr r)^{-1}q^\tr$. When $p=1$ and $\tilde{C}r \neq 0$, the terms containing $\tilde{C}r$ cancel out, and we get
\begin{equation}
    \label{lrcf:eq:update_formula}
    \left.
        \begin{array}{rcl}
            \tilde{V}_1 &=& \sqrt{-2\real{\lam}} \cdot (\tilde{A}^\tr + \lam I)^{-1} \tilde{C}^\tr, \\
            \tilde{Y}_1 &=& I - \frac{1}{2\real{\lam}} (\tilde{V}_1^\tr B)(\tilde{V}_1^\tr B)^\tr, \\
            \tilde{X}_1 &=& \tilde{V}_1 \tilde{Y}_1^{-1} \tilde{V}_1^\tr.
        \end{array}
    \right\}
\end{equation}

\change{The derivation above is valid when $\lam$ is an eigenvalue of the Hamiltonian matrix, 
similarly as in \cite[Theorem 7]{BujB14} which also studies residual Riccati equations related to invariant subspaces
of Hamiltonian matrices. 
Nevertheless, the expression \eqref{lrcf:eq:update_formula} is well-defined even when $\lam$ is not an eigenvalue of the Hamiltonian matrix, 
and for $p>1$ as well. The Hamiltonian argument here serves only as a motivation for introducing \eqref{lrcf:eq:update_formula},
and the following proposition shows that the desired properties of the updated matrix $\Xi + \tilde{X}_1$ still hold for any 
$\lam$ in the left half-plane which is not an eigenvalue of $-\tilde{A}$,
and for all $p$.}

\begin{proposition}
    \label{lrcf:prop:update_is_good}
    Let $\Xi \geq 0$ be such that $\nice{R}(\Xi)=\tilde{C}^\tr \tilde{C} \geq 0$, and let $\lam \in \LHP$ not be an eigenvalue of $-\tilde{A}$. 
    The following holds true for the update matrix $\tilde{X}_1$ as defined in \eqref{lrcf:eq:update_formula}:
    \begin{itemize}
        \item[(a)] $\tilde{X}_1 \geq 0$, i.e. $\Xi + \tilde{X}_1 \geq 0$.
        \item[(b)] $\nice{R}(\Xi + \tilde{X}_1) = \hat{C}^\tr \hat{C} \geq 0$, where 
            $\hat{C}^\tr = \tilde{C}^\tr + \sqrt{-2\real{\lam}} \cdot \tilde{V}_1 \tilde{Y}_1^{-1}$.
    \end{itemize}
\end{proposition}
\begin{proof} \
    \begin{itemize}
        \item[(a)]
            Positive definiteness of $\tilde{Y}_1$ (and then the semi-definiteness of $\tilde{X}_1$ as well) follows directly from $\real{\lam} < 0$.
        \item[(b)]
            Note that $\tilde{A}^\tr \tilde{V}_1 = \sqrt{-2\real{\lam}} \cdot \tilde{C}^\tr  - \lam \tilde{V}_1$, and
            $ (\tilde{V}_1^\tr B)(\tilde{V}_1^\tr B)^\tr = 2\real{\lam} I - 2\real{\lam} \tilde{Y}_1$.
            We use these expressions to obtain:
            \begin{align*}
                \nice{R}(\Xi + \tilde{X}_1)
                    &= \tilde{A}^\tr \tilde{X}_1 + \tilde{X}_1 \tilde{A} + \nice{R}(\Xi) - \tilde{X}_1 BB^\tr \tilde{X}_1 \\
                    &= (\tilde{A}^\tr \tilde{V}_1) \tilde{Y}_1^{-1} \tilde{V}_1^\tr + \tilde{V}_1 \tilde{Y}_1^{-1} (\tilde{A}^\tr \tilde{V}_1)^\tr +
\tilde{C}^\tr \tilde{C} - \tilde{V}_1 \tilde{Y}_1^{-1} (\tilde{V}_1^\tr B) (\tilde{V}_1^\tr B)^\tr  \tilde{Y}_1^{-1} \tilde{V}_1^\tr \\
                    % &= (\sqrt{-2\real{\lam}} \tilde{C}^\tr  - \lam \tilde{V}_1)\tilde{Y}_1^{-1} \tilde{V}_1^\tr + \tilde{V}_1 \tilde{Y}_1^{-1}
% (\sqrt{-2\real{\lam}} \tilde{C}^\tr  - \lam \tilde{V}_1)^\tr - \tilde{V}_1 \tilde{Y}_1^{-1} (2\real{\lam} I - 2\real{\lam} \tilde{Y}_1) \tilde{Y}_1^{-1}
% \tilde{V}_1^\tr \\
                    &= \sqrt{-2\real{\lam}} \cdot \tilde{C}^\tr \tilde{Y}_1^{-1} \tilde{V}_1^\tr + \sqrt{-2\real{\lam}} \cdot \tilde{V}_1 \tilde{Y}_1^{-1}
\tilde{C} + \tilde{C}^\tr \tilde{C} \\
                    & \quad\quad - 2\real{\lam} \tilde{V}_1 \tilde{Y}_1^{-1} \tilde{Y}_1^{-1} \tilde{V}_1^\tr \\
                    &= (\tilde{C}^\tr + \sqrt{-2\real{\lam}}\cdot \tilde{V}_1 \tilde{Y}_1^{-1}) \cdot (\tilde{C}^\tr + \sqrt{-2\real{\lam}}\cdot \tilde{V}_1
\tilde{Y}_1^{-1})^\tr.
            \end{align*}
    \end{itemize}
\end{proof}

\change{%
We are now ready to state the new RADI algorithm.
Starting with the initial approximation 
$X_0=0$ and the residual $\nice{R}(X_0)=C^\tr C$, we continue by selecting a shift $\sigma_k\in\LHP$ and computing the approximation
$X_k = Z_k Y_k^{-1} Z_k^\tr$ with the residual $\nice{R}(X_k) = R_k R_k^\tr$, for $k=1, 2, \ldots$. 
The transition from $X_{k-1}$ to $X_k = X_{k-1} + V_k \tilde{Y}_k^{-1} V_k^\tr$ 
is computed via \eqref{lrcf:eq:update_formula}, adapted to approximate the solution of the residual equation
%\eqref{lrcf:eq:residual_equation} 
with $\Xi = X_{k-1}$, i.e.~ with $\tilde{A} = A - BB^\tr X_{k-1}$ and $\tilde{C}^\tr = R_{k-1}$.
% $$
% 	(A - BB^\tr X_{k-1})^\tr \tilde{X} + \tilde{X} (\underbrace{A-BB^\star X_{k-1}}_{\tilde{A}}) + \nice{R}(X_{k-1}) - \tilde{X} BB^\star \tilde{X} = 0.
% $$
Proposition \ref{lrcf:prop:update_is_good} provides a very efficient update formula for the low-rank factor $R_k$ of the residual.
The whole procedure reduces to the following:}
% , whose factors we suggestively rename in the iteration below, so that $\nice{R}(X_k) = R_k R_k^\tr$.}

\begin{equation}	
	\label{lrcf:eq:qadi-lr}
        \left.
		\begin{array}{rcl}
			R_0 &=& C^\tr, \\
			V_k &=& \sqrt{-2\real{\sigma_k}} \cdot (A^\tr - X_{k-1} BB^\tr + \sigma_k I)^{-1} R_{k-1}, \\
		    \tilde{Y}_k &=& I - \frac{1}{2\real{\sigma_k}} (V_k^\tr B) (V_k^\tr B)^\tr; \quad Y_k = \mb{cc} Y_{k-1} & \\ & \tilde{Y}_k \me, \\
			R_k &=& R_{k-1} + \sqrt{-2\real{\sigma_k}} \cdot V_k \tilde{Y}_k^{-1}, \\
			Z_k &=& \mb{cc} Z_{k-1} & V_k \me.
		\end{array}
        \right\}
        \text{RADI iteration}
\end{equation}
% The parameters $\sigma_k \in \LHP$ can be chosen at will.
Note that any positive semi-definite $X_0$ can be used as an initial
approximation, as long as its residual is positive semi-definite as well, and
its low-rank Cholesky factorization can be computed.  From the derivation of the
RADI algorithm we have that
$$
    X_k = \sum_{i=1}^k V_i \tilde{Y}_i^{-1} V_i^\tr;
$$
in formulation \eqref{lrcf:eq:qadi-lr} of the method we have collected $V_1, \ldots, V_k$ into the matrix $Z_k$, and 
$\tilde{Y}_1, \ldots, \tilde{Y}_k$ into the block-diagonal matrix $Y_k$.

When $p=1$ and all the shifts are chosen as eigenvalues of the Hamiltonian matrix associated with the initial Riccati equation
\eqref{intro:eq:ric}, the update described in Proposition~\ref{lrcf:prop:update_is_good} reduces to 
\cite[Theorem 5]{BujB14}. Thus in that case, 
the RADI algorithm reduces to the invariant subspace approach \eqref{lrcf:eq:ham-ric}.

Furthermore, iteration \eqref{lrcf:eq:qadi-lr} clearly reduces to the low-rank Lyapunov ADI method \eqref{lrcf:eq:lyap-adi-lr} when $B=0$; in that case $Y_k =
I$. The relation to the original qADI iteration \eqref{lrcf:eq:qadi} is not clear unless $p=1$ and the shifts are chosen
as eigenvalues of $\nice{H}$, in which case both of these methods coincide with the invariant subspace approach.
We discuss this further \change{in the following section.} 
% now we turn our attention on how to implement the low-rank RADI algorithm \eqref{lrcf:eq:qadi-lr} efficiently.

%!TEX root = qADI_LR.tex

%%%%%%%%%%%%%%%%%%%%%%%%%%%%%%%%%%%%%%%%%%%%%%%%%%%%%%%% 
% .. Equivalences with other Riccati methods 
%%%%%%%%%%%%%%%%%%%%%%%%%%%%%%%%%%%%%%%%%%%%%%%%%%%%%%%%
\section{Equivalences with other Riccati methods}
\label{equiv}

In this section we prove that all Riccati solvers introduced in Section \ref{lrcf} in fact compute exactly the same iterations,
which we will refer to as the Riccati ADI iterations in the remaining text.
\change{This result is collected in Theorem \ref{equiv:tm}; 
we begin with a simple technical lemma that provides different representations of the residual factor.}

\begin{lemma}
    \label{lrcf:tm:residual}
    Let
    \begin{align*}
        R_k^{(1)} &= \frac{1}{\sqrt{-2\real{\sigma_k}}} \cdot (A^\tr - X_k G - \conj{\sigma_k}I ) V_k, \\
        R_k^{(2)} &= \frac{1}{\sqrt{-2\real{\sigma_{k+1}}}} \cdot (A^\tr - X_k G + \sigma_{k+1}I ) V_{k+1}.
    \end{align*}
    Then $R_k^{(1)} = R_k^{(2)} = R_k$.
\end{lemma}
\begin{proof}
    From the definition of the RADI iteration \eqref{lrcf:eq:qadi-lr} it is obvious that $R_k^{(2)} = R_k$.
    Using $X_k = X_{k-1} + V_k \tilde{Y}_k^{-1} V_k^\tr$, 
    and $V_k^\tr G V_k = 2\real{\sigma_k} I - 2\real{\sigma_k} \tilde{Y}_k$, we have
    \begin{align*}
        R_k 
            &= R_{k-1} + \sqrt{-2\real{\sigma_k}} V_k \tilde{Y}_k^{-1} \\
            &= \frac{1}{\sqrt{-2\real{\sigma_{k}}}} \cdot (A^\tr - X_{k-1} G + \sigma_{k}I ) V_{k} + \sqrt{-2\real{\sigma_k}} V_k \tilde{Y}_k^{-1} \\
            &= \frac{1}{\sqrt{-2\real{\sigma_{k}}}} \cdot (A^\tr - X_{k} G - \conj{\sigma_{k}}I ) V_{k} + \frac{1}{\sqrt{-2\real{\sigma_{k}}}} V_k
\tilde{Y}_k^{-1} V_k^\tr G V_k \\
            & \quad\quad - \sqrt{-2\real{\sigma_k}} V_k + \sqrt{-2\real{\sigma_k}} V_k \tilde{Y}_k^{-1} \\
            &= R^{(1)}_{k}.
    \end{align*}
\end{proof}

\begin{theorem}
	\label{equiv:tm}
    If the initial approximation in all algorithms is zero, and the same shifts are used, then for all $k$,
    $$
        X_k = \qX = \cX.
    $$
    If $\rank C = 1$ and the shifts are equal to distinct eigenvalues of $\nice{H}$, then for all $k$,
    $$
        X_k = \qX = \cX = \hX.
    $$
\end{theorem}
\begin{proof}
	We first use induction to show that $X_k = \qX$, for all $k$. \\
	\

	Assume that $X_{k-1} = \qX[k-1]$.
	We need to show that $X_k = X_{k-1} + V_k \YY_k^{-1} V_k^\tr$ satisfies the defining equality \eqref{lrcf:eq:qadi} of the {qADI} iteration, i.e. that
	\begin{equation}
		\label{lrcf:eq:thm2:temp1}
		(A^\tr + \sigma_{k} I - \qX[k-1/2] G) (X_{k-1} + V_k \YY_k^{-1} V_k^\tr) = -Q - \qX[k-1/2] (A - \sigma_{k} I),
	\end{equation}
	where
	\begin{equation}
		\label{lrcf:eq:thm2:temp2}
		\qX[k-1/2] (A+\conj{\sigma_{k}} I - G X_{k-1}) = -Q - (A^\tr - \conj{\sigma_{k}} I) X_{k-1}.
	\end{equation}
	First, note that \eqref{lrcf:eq:thm2:temp2} can be rewritten as
	$$
		A^\tr X_{k-1} + \qX[k-1/2] A + Q - \qX[k-1/2] G X_{k-1} + \conj{\sigma_k}(\qX[k-1/2] - X_{k-1}) = 0.
	$$
	Subtracting this from the expression for the Riccati residual,
	$$
		A^\tr X_{k-1} + X_{k-1} A + Q - X_{k-1} G X_{k-1} = \nice{R}(X_{k-1}),
	$$
	we obtain
	\begin{equation}
		\label{lrcf:eq:thm2:temp3}
		\qX[k-1/2] - X_{k-1} = -\nice{R}(X_{k-1}) \cdot (A - GX_{k-1} + \conj{\sigma_k}I)^{-1}.	
	\end{equation}
	Equation \eqref{lrcf:eq:thm2:temp1} can be reorganized as
	\begin{align*}
		(A^\tr &+ \sigma_{k} I ) (X_{k-1} + V_k \YY_k^{-1} V_k^\tr) - \qX[k-1/2] G V_k \YY_k^{-1} V_k^\tr \\
			&= -Q - \qX[k-1/2] (A +\conj{\sigma_{k}} I - G X_{k-1} ) +  2\real{\sigma_{k}} \qX[k-1/2].
		% (A^\tr + \sigma_{k} I - X_{k-1}G ) (X_{k-1} + V_k \YY_k^{-1} V_k^\tr) 
		%     &+ X_{k-1}G  (X_{k-1} + V_k \YY_k^{-1} V_k^\tr) 
		%      - \qX[k-1/2] G V_k \YY_k^{-1} V_k^\tr	\\
		%     &= -Q - \qX[k-1/2] (A +\conj{\sigma_{k}} I - G X_{k-1} ) +  2\real{\sigma_{k}} \qX[k-1/2].	
	\end{align*}
	Replace the second term on the right-hand side with the right-hand side of \eqref{lrcf:eq:thm2:temp2}.
	Thus, it remains to prove
	$$
		(A^\tr + \sigma_{k} I ) (X_{k-1} + V_k \YY_k^{-1} V_k^\tr) - \qX[k-1/2] G V_k \YY_k^{-1} V_k^\tr
			= (A^\tr - \conj{\sigma_{k}} I) X_{k-1} +  2\real{\sigma_{k}} \qX[k-1/2],
	$$
	or after some rearranging, and by using \eqref{lrcf:eq:thm2:temp3},
	\begin{align}
		(A^\tr &- X_{k-1}G + \sigma_k I)V_k \YY_k^{-1} V_k^\tr \nonumber \\
			&= (\qX[k-1/2] - X_{k-1})  \cdot (2\real{\sigma_k}I + G V_k \YY_k^{-1} V_k^\tr) \nonumber \\ 
			&= -\nice{R}(X_{k-1}) \cdot (A - GX_{k-1} + \conj{\sigma_k}I)^{-1} \cdot (2\real{\sigma_k}I + G V_k \YY_k^{-1} V_k^\tr). \label{lrcf:eq:thm2:temp4}
	\end{align}
	Next we use the expression $\nice{R}(X_{k-1}) = R_{k-1}^{(2)}(R_{k-1}^{(2)})^\tr$ of Lemma \ref{lrcf:tm:residual}.
	The right-hand side of \eqref{lrcf:eq:thm2:temp4} is, thus, equal to
	$$
		\frac{1}{2\real{\sigma_{k}}} \cdot (A^\tr - X_{k-1} G + \sigma_{k}I ) V_{k} V_k^\tr \cdot
		(2\real{\sigma_k}I + G V_k \YY_k^{-1} V_k^\tr),
	$$
	which turns out to be precisely the same as the left-hand side of \eqref{lrcf:eq:thm2:temp4} once we use the identity
	$$
		V_k^\tr G V_k = 2\real{\sigma_k}I - 2\real{\sigma_k} \YY_k.
	$$
	This completes the proof of $X_k = \qX$.
	\\
	\

	Next, we use induction once again to show $\cX = \qX$, for all $k$.
	For $k=0$, the claim is trivial; assume that $\cX[k-1] = \qX[k-1]$ for some $k \geq 1$.
	To show that $\cX = \qX$, let us first multiply \eqref{lrcf:eq:cay} by $\nice{H}-\sigma_k I$ from the left:
	\begin{equation}
		\label{eq:CHSIrec1}
		\mb{cc} A - \sigma_k I & G \\ Q & -A^\tr - \sigma_k I \me
			\mb{c} M_k \\ N_k \me
		= 
		\mb{cc} A + \conj{\sigma_k} I & G \\ Q & -A^\tr + \conj{\sigma_k} I \me
			\mb{c} I \\ -\cX[k-1] \me
		=:
		\mb{c} M_{k-1/2} \\ N_{k-1/2} \me,
	\end{equation}
	and suggestively introduce $\cX[k-1/2] := -N_{k-1/2} M_{k-1/2}^{-1}$. We thus have
	\begin{eqnarray}
		A + \conj{\sigma_k} I - G \cX[k-1]         & = & M_{k-1/2}, \\
		Q - (-A^\tr + \conj{\sigma_k} I) \cX[k-1] & = & N_{k-1/2},
	\end{eqnarray}
	and
	$$
		\cX[k-1/2] (A + \conj{\sigma_k} I - G \cX[k-1]) 
			= -N_{k-1/2} M_{k-1/2}^{-1} M_{k-1/2} 
			= -Q - (A^\tr - \conj{\sigma_k} I) \cX[k-1].
	$$
	This is the same relation as the one defining $\qX[k-1/2]$, and thus $\cX[k-1/2] = \qX[k-1/2]$.
	Next, equating the leftmost and the rightmost matrix in \eqref{eq:CHSIrec1}, it follows that
	\begin{eqnarray}
		(A - \sigma_k I) M_k + GN_k        & = & M_{k-1/2}, \label{eq:CHSIrec2} \\
		Q M_k + (-A^\tr - \sigma_k I) N_k  & = & N_{k-1/2}. \label{eq:CHSIrec3}
	\end{eqnarray}
	Multiply \eqref{eq:CHSIrec3} from the right by $M_k^{-1}$ to obtain
	\begin{equation}
		\label{eq:CHSIrec4}
		(A^\tr + \sigma_k I) \cX = -Q + N_{k-1/2}M_k^{-1},
	\end{equation}
	and multiply \eqref{eq:CHSIrec2} from the left by $\cX[k-1/2]$ and from the right by $M_k^{-1}$ to get
	\begin{equation}
		\label{eq:CHSIrec5}
		-\cX[k-1/2]G\cX = -\cX[k-1/2] (A-\sigma_k I) - N_{k-1/2}M_k^{-1}.
	\end{equation}
	Adding \eqref{eq:CHSIrec4} and \eqref{eq:CHSIrec5} yields
	$$
		(A^\tr + \sigma_k I - \cX[k-1/2]G) \cX = -Q - \cX[k-1/2] (A-\sigma_k I),
	$$
	which is the same as the defining equation for $\qX$. Thus $\cX = \qX$, so both the Cayley subspace iteration and the qADI iteration
	generate the same sequences.
	\\
	\

	In the case of $\rank{C}=1$ and shifts equal to the eigenvalues of $\nice{H}$, the equality $\hX = \qX$ is already shown in \cite{BujB14}.
	Equality among the iterates generated by the other methods is a special case of what we have proved above.
\end{proof}

\change{It is interesting to observe that \cite{LinSim13} also provides a low-rank variant of the Cayley subspace iteration algorithm: 
there, formulas for updating the factors of $\cX = \cZ (\cY)^{-1} (\cZ)^\tr$, where $\cZ \in \Complex^{n \times pk}$ and $\cY \in \Complex^{pk \times pk}$, are
given.
The contribution of \cite{MasOR14} was to show that the same formulas can be derived from a control-theory point of view.}
The main difference in comparison to our RADI variant of the low-rank Riccati ADI iterations 
\change{is that, in order to compute $\cZ$, one uses} the matrix $(A^\tr +{\sigma_k}I)^{-1}$, instead of $(A^\tr-X_{k-1}G+{\sigma_k}I)^{-1}$, when computing
$Z_k$. 
This way, the need for using the Sherman--Morrison--Woodbury formula is avoided.
\change{However, as a consequence, the matrix $\cY$ looses the block-diagonal structure, and 
its update formula becomes much more involved.}
Also, it is very difficult to derive a version of the algorithm that would use real arithmetic.
Another disadvantage is the computation of the residual: along with $\cZ$ and $\cY$, one needs to
maintain a QR-factorization of the matrix $[C^\tr \;\; A^\tr \cZ \;\; \cZ ]$, 
which adds significant computational complexity to the algorithm.
% Both \cite{LinSim13} and \cite{MasOR14} notice that setting $B=0$ reduces the Cayley subspace iteration algorithm to Lyapunov ADI method; they do not provide
% any
% connection to the qADI algorithm nor the Hamiltonian invariant subspace method.
\\

Each of these different statements of the same Riccati ADI algorithm may contribute when studying theoretical properties of the iteration.
For example, directly from our definition \eqref{lrcf:eq:qadi-lr} of the RADI iteration it is obvious that
$$
	0 \leq X_1 \leq X_2 \leq \ldots \leq X_k \leq \ldots \leq X.
$$
% where the relation $\leq$ is defined by the L\"{o}wner partial order. 
Also, the fact that the residual matrix $\nice{R}(X_k)$
is low-rank and its explicit factorization follows naturally from our approach.
On the other hand, approaching the iteration from the control theory point of view as in \cite{MasOR14} is more suitable for
proving that the non-Blaschke condition for the shifts,
$$
	\sum_{k=1}^{\infty} \frac{\real{\sigma_k}}{1 + |\sigma_k|^2} = -\infty,
$$
is sufficient for achieving the convergence when $A$ is stable, i.e.
$$
	\lim_{k \to \infty} X_k = X.
$$

We conclude this section by noting a relation between the Riccati {ADI} iteration and the rational Krylov subspace method \cite{SimSM13a}.
It is easy to see that the {RADI} iteration also uses the rational Krylov subspaces as the basis for approximation.
\change{This fact also follows from the low-rank formulation for $\cX$ as given in \cite{LinSim13}, 
so we only state it here without proof.}

\begin{proposition}
	For a matrix $M$, (block-)vector $v$ and a tuple $\overrightarrow{\sigma_k} = (\sigma_1, \ldots, \sigma_k) \in \LHP^k$, let
	$$
		\nice{K}(M, v, \overrightarrow{\sigma_k}) = \spann{(M+\sigma_1 I)^{-1}v, (M+\sigma_2 I)^{-1}v, \ldots, (M+\sigma_k I)^{-1}v}
	$$
	denote the rational Krylov subspace generated by $M$ and the initial vector $v$.
	Then the columns of $X_k$ belong to $\nice{K}(A^\tr, C^\tr, \overrightarrow{\sigma_k})$.
\end{proposition}
% \begin{proof}
% 	Since $X_k = Z_k Y_k^{-1} Z_k^\tr$, it follows that the columns of $X_k$ belong to $\spann{Z_k}$.
% 	We show that these lie in $\nice{K}(A^\tr, C^\tr, \overrightarrow{\sigma_k})$ by using induction.
% 	For $k=1$,
% 	$$
% 		V_1 = \sqrt{-2\real{\sigma_1}} \cdot (A^\tr + \sigma_1 I)^{-1} C^\tr,
% 	$$
% 	so the claim is obvious. Suppose that $\spann{Z_{k-1}}$ is a subspace of the rational Krylov subspace 
% 	$\nice{K}(A^\tr, C^\tr, \overrightarrow{\sigma_{k-1}})$ generated by the shifts $\sigma_1, \ldots, \sigma_{k-1}$.
% 	Since
% 	$$
% 		V_k = \alpha_k V_{k-1} + \beta_k (A^\tr - X_{k-1}G + \sigma_k I)^{-1} V_{k-1}
% 	$$
% 	for some complex numbers $\alpha_k$ and $\beta_k$, after using the Sherman--Morrison--Woodbury formula we have that
% 	\begin{align*}
% 		V_k &= \alpha_k V_{k-1} + \beta_k (A^\tr + \sigma_k I)^{-1} V_{k-1} \\
% 			& \quad\quad + \beta_k (A^\tr + \sigma_k I)^{-1} Z_{k-1} Y_{k-1}^{-1}Z_{k-1}^\tr (I - GA^{-\tr} X_{k-1})^{-1} G(A^\tr + \sigma_k I)^{-1}
% V_{k-1} \\
% 			&\in \nice{K}(A^\tr, C^\tr, \overrightarrow{\sigma_{k-1}}) + (A^\tr + \sigma_k I)^{-1} \nice{K}(A^\tr, C^\tr, \overrightarrow{\sigma_{k-1}}) \\
% 			&\subseteq \nice{K}(A^\tr, C^\tr, \overrightarrow{\sigma_{k}}).	
% 	\end{align*}
% 	Here the element symbol is understood column-wise.\qed
% \end{proof}

From the proposition we conclude the following: if $U_k$ contains a basis for the rational Krylov subspace $\nice{K}(A^\tr, C^\tr, \overrightarrow{\sigma_k})$,
then both the approximation $\rX$ of the Riccati solution obtained by the rational Krylov subspace method and the approximation $\qX$ obtained by 
a Riccati {ADI} iteration satisfy
$$
	\rX = U_k \rY U_k^\tr, \quad \qX = U_k \qY U_k^\tr,
$$
for some matrices $\rY, \qY \in \Complex^{pk \times pk}$. The columns of both $\rX$ and $\qX$ belong to the same subspace, and the only difference between
the methods is the choice of the linear combination of columns of $U_k$, i.e. the choice of the small matrix $Y_k$.
The rational Krylov subspace method \cite{SimSM13a} generates its $\rY$ by solving the projected Riccati equation, while 
the Riccati {ADI} methods do it via direct formulas such
as the one in \eqref{lrcf:eq:qadi-lr}.

%!TEX root = qADI_LR.tex

\section{Implementation aspects of the RADI algorithm}
\label{impl}
There are several issues with the iteration \eqref{lrcf:eq:qadi-lr}, stated as is, that should be addressed when designing an efficient computational
routine: how to decide when the iterates $X_k$ have converged,
how to solve linear systems with matrices $A^\tr - X_{k-1}G + \sigma_k I$, and how to minimize the usage of complex arithmetic.
\change{In this section we also discuss the various shift selection strategies.}
% \\

% {\bf Computing the residual and the stopping criterion.}
\subsection{Computing the residual and the stopping criterion.}
Tracking the progress of the algorithm and deciding when the iterates have converged is 
\change{very simple, and can be computed cheaply thanks to the expression 
$
    \|\nice{R}(X_k)\| = \|R_k R_k^\tr\| = \|R_k^\tr R_k\|.
$
This is an advantage compared to the Cayley subspace iteration, where computing $\|\nice{R}(X_k)\|$ is more expensive because a low-rank factorization
along the lines of Proposition~\ref{lrcf:prop:update_is_good}~(b) is currently not known.}
%  thanks to the connection
% between the iteration matrices $R_k \in \Complex^{n \times p}$ and the \change{residuals $\nice{R}(X_k) = R_k R_k^\tr$:
% Computing the norm of $\nice{R}(X_k)$ is done in the following way:
% The reordering of the factors in the last step is valid for both the 2-norm and the Frobenius norm. By applying this trick, one only has to
% compute the norm of a $p\times p$ matrix instead of the original $n\times n$ matrix.
The RADI iteration is stopped once the residual norm has decreased sufficiently relative to the initial residual norm $\|C C^\tr \|$
of the approximation $X_0=0$. 
% Similar tricks have been used for the low-rank Lyapunov ADI method
% \cite{PanW13a,morBenKS13} as well.
% \\
% \

% {\bf Solving linear systems in RADI.}
\subsection{Solving linear systems in RADI}\label{ssec:linsysRADI}
During the iteration, one has to evaluate the expression $(A^\tr - X_{k-1}G + \sigma_k I)^{-1} R_{k-1}$.
Here the matrix $A$ is assumed to be sparse, while $X_{k-1}G = (X_{k-1}B) B^\tr$ is low-rank.
There are different options on how to solve this linear system; if one wants to use a direct sparse solver, the initial expression
can be adapted by using the Sherman--Morrison--Woodbury (SMW) formula \cite{GolV13}.
We introduce the \change{approximate feedback} matrix $K_{k} := X_k B$ and update it during the RADI iteration: 
$K_k = K_{k-1} + (V_k \tilde{Y}_k^{-1}) (V_k^\tr B)$.
Note that $V_k \tilde{Y}_k^{-1}$ also appears in the update of the residual factor, and that $V_k^\tr B$ appears in the computation of $\tilde{Y}_k$, so both
have to be computed only once.
The initial expression is rewritten as
\begin{align*}
    (A^\tr- K_{k-1}B^\tr + \sigma_k I)^{-1} R_{k-1}&=L_k+N_k(I_m-B^\tr N_k)^{-1}K_{k-1}^\tr L_k,\\
    [L_k,N_k]&=(A^\tr+\sigma_kI)^{-1}[R_{k-1},K_{k-1}].
%   (A^\tr + \sigma_k I)^{-1} R_{k-1} + \\
%         & \quad\quad (A^\tr + \sigma_k I)^{-1}K_{k-1} (I - B^\tr (A^\tr + \sigma_k I)^{-1} K_{k-1})^{-1} B^\tr (A^\tr + \sigma_k I)^{-1} R_{k-1}.
\end{align*}
\change{Thus, in each RADI step one needs to solve a linear system with the coefficient matrix $A^\tr + \sigma_k I$ and $p+m$ right hand sides.
A very similar technique is used in the low-rank Newton ADI solver for the Riccati equation \cite{BenS10,BenLP08,Saa09,BenHW15}.
In the equivalent Cayley subspace iteration, linear systems defined by $A^\tr + \sigma_k I$ and only $p$ right hand sides have to be solved, which 
makes their solution less expensive than their counterparts in RADI.} 
\\
\

\begin{zAlgorithmFloat}
    \KwIn{matrices $A \in \Real^{n \times n}$, $B \in \Real^{n \times m}$, $C\in \Real^{p \times n}$.}
    \KwOut{approximation $X \approx Z Y^{-1} Z^\tr$ for the solution of $A^\tr X + XA + C^\tr C - X BB^\tr X =0$.}

    \vspace*{0.3cm}
    $R = C^\tr$; $K = 0$; $Y = [\;]$\;
    \While{$\|R^\tr R\| \geq tol \cdot \|C C^\tr\|$}
    {
        Obtain the next shift $\sigma$\;
        \eIf{first pass through the loop}
        {
            $Z = V = \sqrt{-2\real{\sigma}} \cdot (A^\tr + \sigma I)^{-1} R$\;
        }
        {
            
            $V = \sqrt{-2\real{\sigma}} \cdot (A^\tr - K B^\tr + \sigma I)^{-1} R$; \tcp{Use SMW if necessary}
            $Z = [Z \;\; V]$\;        
        }
        
        $\tilde{Y} = I - \frac{1}{2\real{\sigma}} \cdot (V^\tr B) (V^\tr B)^\tr$;
        $Y = \smb Y & \\ & \tilde{Y} \sme$;

        $R = R + \sqrt{-2 \real{\sigma}} \cdot (V \tilde{Y}^{-1})$\;
        $K = K + (V \tilde{Y}^{-1}) \cdot (V^\tr B)$;
    }    
    \caption{The RADI iteration using complex arithmetic.}\label{lrcf:alg:qadi-complex}
\end{zAlgorithmFloat}
\change{The RADI algorithm, implementing the techniques described above, is listed in Algorithm~\ref{lrcf:alg:qadi-complex}.
Note that, if only the feedback matrix $K$ is of interest, e.g. if the CARE
arises from an optimal control problem, there is no need to store the whole
low-rank factors $Z$, $Y$ since Algorithm~\ref{lrcf:alg:qadi-complex} requires
only the latest blocks to continue. This is again similar to the low-rank
Newton ADI solver \cite{BenLP08}, and not possible in the current version of the Cayley subspace iteration.}
% {\bf \change{Reducing} the use of complex arithmetic.}
\subsection{\change{Reducing} the use of complex arithmetic.}
To increase the efficiency of Algorithm~\ref{lrcf:alg:qadi-complex}, we \change{reduce} the use of complex arithmetic. We do so by taking shift $\sigma_{k+1} =
\conj{\sigma_k}$
immediately after the shift $\sigma_k \in \Complex \setminus \Real$ has been used, and by merging these two consecutive RADI steps into a single one.
This entire procedure will have only one operation involving complex matrices: a linear solve with the matrix $A^\tr - X_{k-1}G + \sigma_k I$ to
compute $V_k$. 
There are two key observations to be made here. First, by modifying the iteration slightly, one can ensure that
the matrices $K_{k+1}$, $R_{k+1}$, and $Y_{k+1}$ are real and can be computed by using real arithmetic only, as shown in the
upcoming technical proposition. Second, there is no need to compute $V_{k+1}$ at all to proceed with the iteration: the next matrix
$V_{k+2}$ will once again be computed by using the residual $R_{k+1}$, the same way as in Algorithm~\ref{lrcf:alg:qadi-complex}.
\begin{proposition}
    \label{lrcf:prop:qadi-real}
    Let $X_{k-1} = Z_{k-1} Y_{k-1}^{-1} Z_{k-1}^\tr \in \Real^{n \times n}$ denote the Riccati approximate solution computed after $k-1$ RADI steps.
    Assume that $R_{k-1},~V_{k-1},~K_{k-1}$ are real and that $\sigma := \sigma_{k} = \conj{\sigma_{k+1}} \in \Complex \setminus \Real$.
    Let: 
    % \textcolor{blue}{Pad: Maybe its better to use $V_R, V_I$ in order to avoid confusion with the $V_k$. 
    % Zvonimir: OK, I've changed it throughout. NO: it is very difficult to see the difference between $V_1$ and $V_I$. Changed back to 
    % $V_{r}$ and $V_{i}$.}
    \begin{align*}
        V_{r} &= (\real{V_k})^\tr B, \quad V_{i} = (\imag{V_k})^\tr B, \\
        F_1 &= \mb{c} -\real{\sigma} V_{r} - \imag{\sigma} V_{i} \\ \imag{\sigma}V_{r} - \real{\sigma}V_{i} \me, \quad
        F_2 = \mb{c} V_{r} \\ V_{i} \me, \quad
        F_3 = \mb{c} \imag{\sigma} I_p \\ \real{\sigma}I_p \me.
    \end{align*}
    Then $X_{k+1} = Z_{k+1} Y_{k+1}^{-1} Z_{k+1}^\tr$, where:
    \begin{align*}
        Z_{k+1} &= [Z_{k-1} \;\; \real{V_k} \;\; \imag{V_k}], \\
        Y_{k+1} &= \mb{cc} Y_{k-1} & \\ & \hat{Y}_{k+1} \me, \\
        \hat{Y}_{k+1} &= \mb{cc} I_p & \\ & 1/2 I_p \me - \frac{1}{4|\sigma|^2 \real{\sigma}} F_1 F_1^\tr - \frac{1}{4\real{\sigma}} F_2 F_2^\tr -
\frac{1}{2|\sigma|^2} F_3 F_3^\tr.
    \end{align*}
%     \textcolor{blue}{Pad: could we merge $\mb{cc} I_p & \\ & 1/2 I_p \me$ and $\frac{1}{2|\sigma|^2} F_3 F_3^\tr$, $F_3 = \mb{c} \imag{\sigma} I_p \\
% \real{\sigma}I_p \me$? Zvonimir: I don't see how, the formula would be more complicated. I would prefer not to expand the expression into a 2x2 matrix.\\
% Pad: ok, it would be something like $\frac{1}{2(1+2\beta^2)}\mb{cc}1+2\beta^2 &-\beta \\
% -\beta& 1\me$,$\beta=\frac{\real{\sigma}}{\imag{\sigma}}$
% Zvonimir: Let's leave it as is, the referees didn't complain about this.}
    The residual factor $R_{k+1}$ and the matrix $K_{k+1}$ can be computed as
    \begin{align*}
        R_{k+1} &= R_{k-1} + \sqrt{-2\real{\sigma}} \left( [\real{V_k} \;\; \imag{V_k}] \hat{Y}_{k+1}^{-1} \right) (:, 1:p), \\
        K_{k+1} &= K_{k-1} + [\real{V_k} \;\; \imag{V_k}] \hat{Y}_{k+1}^{-1} \mb{c} V_{r} \\ V_{i} \me. 
    \end{align*}
\end{proposition}
\begin{proof}
    The \change{basic} idea of the proof is similar to \cite[Theorem 1]{BenKS13}; 
    however, there is a major complication involving the matrix $Y$, which 
    in the Lyapunov case is simply equal to the identity.
    Due to the technical complexity of the proof, we only display key intermediate results.
    To simplify notation, we use indices $0$, $1$, $2$ instead of $k-1$, $k$, $k+1$, respectively.
    %Let us assume that $R_0$ is also a real matrix; we will show that $R_2$ is then also real.
    \\
    We start by taking the imaginary part of the defining relation for $R_0^{(2)}=R_0$ in Lemma~\ref{lrcf:tm:residual}:
    $$
        0 = (A^\tr - X_0 G + \real{\sigma}I) \cdot \imag{V_1} + \imag{\sigma} \cdot \real{V_1};
    $$
    thus
    $$
        V_1 = \real{V_1} + \ii \imag{V_1} = \frac{-1}{\imag{\sigma}} \underbrace{(A^\tr - X_0 G + \conj{\sigma} I)}_{A_0} \imag{V_1}.
    $$
%     \todo{P{\"a}d: Will the derivations become easier if we use the $V-R$-representations, i.e.,  $V_2=(A^\tr-X_1G+\overline{\sigma}I)^{-1}R_1$? For the
% Lyapunov ADI this is the case. Zvonimir: Not sure, it seems to be the same.}
    We use this and the Sherman--Morrison--Woodbury formula to compute $V_2$:
    \begin{align}
        V_2 
            &= \sqrt{-2\real{\sigma}} (A^\tr - X_1 G + \conj{\sigma} I)^{-1} R_1 \nonumber \\
            &= \sqrt{-2\real{\sigma}} (A^\tr - X_1 G + \conj{\sigma} I)^{-1} \cdot \frac{1}{\sqrt{-2\real{\sigma}}} \cdot (A^\tr - X_1 G - \conj{\sigma}I ) V_1
\nonumber \\
            &= V_1 - 2\conj{\sigma} (A^\tr - X_1 G + \conj{\sigma}I)^{-1} V_1 \nonumber \\
            &= V_1 - 2\conj{\sigma} \left( A_0 - V_1 \hY_1^{-1} (V_1^\tr G) \right)^{-1} V_1 \nonumber \\
            &= V_1 - 2\conj{\sigma} ( A_0^{-1}V_1 + A_0^{-1} V_1 ( \hY_1 - (V_1^\tr G)A_0^{-1}V_1)^{-1} (V_1^\tr G) A_0^{-1}V_1 ) \nonumber \\
            &= V_1 + 2\conj{\sigma} \imag{V_1} (\underbrace{\imag{\sigma} \hY_1 + (V_1^\tr B)V_{i}^\tr }_{S})^{-1} \hY_1; \label{lrcf:eq:prop-temp1}
    \end{align}
    \change{where we used $\imag{\sigma} A_0^{-1} V_1 = -\imag{V_1}$ to obtain the last line.}
    Next,
    \begin{align*}    
        X_2 &= X_0 + \smb V_1 & V_2 \sme \smb \hY_1^{-1} & \\ & \hY_2^{-1} \sme \smb V_1 & V_2 \sme^\tr \\
            &= X_0 + \smb \real{V_1} & \imag{V_1} \sme 
                \underbrace{\smb I & I \\ \ii I & \ii I + 2\conj{\sigma} S^{-1} \hY_1 \sme}_T
                \smb \hY_1^{-1} & \\ & \hY_2^{-1} \sme 
                \left(\smb \real{V_1} & \imag{V_1} \sme   \smb I & I \\ \ii I & \ii I + 2\conj{\sigma} S^{-1} \hY_1 \sme  \right)^\tr \\
            &= X_0 + \smb \real{V_1} & \imag{V_1} \sme 
                \big( \underbrace{T^{-\tr} \smb \hY_1 & \\ & \hY_2 \sme T^{-1}}_{\hat{Y}_2} \big)^{-1}
                \smb \real{V_1} & \imag{V_1} \sme^\tr.
    \end{align*}
    We first compute $T^{-1}$ by using the SMW formula once again:
    \begin{align*}    
        T^{-1}
            &= \left( \smb & I \\ \ii I & \ii I \sme + \smb I & \\ & S^{-1} \sme \smb I & \\ & 2\conj{\sigma}\hY_1 \sme \right)^{-1} \\
            &= \mb{cc} I + \frac{\ii}{2\conj{\sigma}} \hY_1^{-1}S & \frac{-1}{2\conj{\sigma}}\hY_1^{-1} S \\
                    \frac{-\ii}{2\conj{\sigma}} \hY_1^{-1}S & \frac{1}{2\conj{\sigma}} \hY_1^{-1}S
               \me,
    \end{align*}
    and, applying the congruence transformation with $T^{-1}$ to $\smb \hY_1 & \\ & \hY_2 \sme$ yields
    % \todo{\\ Any ideas on how to truncate this to fit in one row?}
    $$
        \hat{Y}_2
            = \mb{cc}
                    \hY_1 + \frac{\ii}{2\conj{\sigma}} S - \frac{\ii}{2{\sigma}} S^\tr +   \frac{1}{4|\sigma|^2} S^\tr \hY_1^{-1}S +   \frac{1}{4|\sigma|^2}
S^\tr \hY_1^{-1} \hY_2 \hY_1^{-1} S
            ~~~~&~~~\frac{-1}{2\conj{\sigma}} S                               + \frac{\ii}{4|\sigma|^2} S^\tr \hY_1^{-1}S + \frac{\ii}{4|\sigma|^2}
S^\tr \hY_1^{-1} \hY_2 \hY_1^{-1} S \\
                                                            \frac{-1}{2{\sigma}} S^\tr - \frac{\ii}{4|\sigma|^2} S^\tr \hY_1^{-1}S - \frac{\ii}{4|\sigma|^2}
S^\tr \hY_1^{-1} \hY_2 \hY_1^{-1} S
                  &                                                                        \frac{1}{4|\sigma|^2} S^\tr \hY_1^{-1}S +   \frac{1}{4|\sigma|^2}
S^\tr \hY_1^{-1} \hY_2 \hY_1^{-1} S
               \me.
    $$
    By using \eqref{lrcf:eq:prop-temp1}, it is easy to show
    \begin{align*}
        \hY_2 
            &= \change{I - \frac{1}{2\real{\sigma}} (V_2^\tr B) (V_2^\tr B)^\tr} \\
            &= -\hY_1 - \frac{1}{2\real{\sigma}} \left( 
                -2\sigma \imag{\sigma} \hY_1S^{-\tr}\hY_1 - 2\conj{\sigma}\imag{\sigma}\hY_1 S^{-1}\hY_1
                +4|\sigma|^2 \hY_1S^{-\tr} V_{i} V_{i}^\tr S^{-1} \hY_1
            \right).    
    \end{align*}
    Inserting this into the formula for $\hat{Y}_2$, all terms containing inverses of $\hY_1$ and $S$ cancel out. 
    By rearranging the terms that do appear in the formula, we get the expression from the
    claim of the proposition.

    Deriving the formulae for $R_2$ and $K_2$ is straightforward:
    \begin{align*}
        K_2 
            &= K_0 + \smb V_1 & V_2 \sme \smb \hY_1^{-1} & \\ & \hY_2^{-1} \sme \smb V_1 & V_2 \sme^\tr B\\
            &= K_0 + \left( \smb \real{V_1} & \imag{V_1} \sme \hat{Y}_2^{-1} \right) \left( \smb \real{V_1} & \imag{V_1} \sme^\tr B \right), \\
        R_2 
            &= R_0 + \sqrt{-2\real{\sigma}} \smb V_1 & V_2 \sme \smb \hY_1^{-1} \\ \hY_2^{-1} \sme \\
            &= R_0 + \sqrt{-2\real{\sigma}} \smb \real{V_1} & \imag{V_1} \sme \left( T \smb \hY_1^{-1} \\ \hY_2^{-1} \sme \right) \\
            &= R_0 + \sqrt{-2\real{\sigma}} \smb \real{V_1} & \imag{V_1} \sme \hat{Y}_2^{-1} (:, 1:p),
    \end{align*}
    where the last line is due to the structure of the matrix $T$. Since the matrix $\hat{Y}_2$ is real, so are $K_2$ and $R_2$.
\end{proof}
% Algorithm~\ref{lrcf:alg:qadi-real} shows the final implementation, taking into account Proposition \ref{lrcf:prop:qadi-real}.
% \\
\begin{zAlgorithmFloat}
    \KwIn{matrices \change{$A, E \in \Real^{n \times n}$}, $B \in \Real^{n \times m}$, $C\in \Real^{p \times n}$.}
    \KwOut{approximation $X \approx Z Y^{-1} Z^\tr$ for the solution of the generalized equation
        \change{$A^\tr XE + E^\tr X A + C^\tr C - E^\tr X BB^\tr XE = 0$}. The matrices $Z$ and $Y$ are real.}

    \vspace*{0.3cm}
    $R = C^\tr$; $K = 0$; $Y = [\;]$; $Z = [\;]$\;
    \While{$\|R^\tr R\| \geq tol \cdot \|C C^\tr\|$}
    {
        Obtain the next shift $\sigma$\;
        \vspace*{0.3cm}    

        \eIf{first pass through the loop}
        {
            $V = \sqrt{-2\real{\sigma}} \cdot (A^\tr + \sigma \change{E^\tr})^{-1} R$\;
        }
        {
            $V = \sqrt{-2\real{\sigma}} \cdot (A^\tr - K B^\tr + \sigma \change{E^\tr})^{-1} R$; \tcp{Use SMW if necessary}
        }
        \vspace*{0.3cm}

        \eIf{$\sigma \in \Real$}
        {
            $Z = [Z \;\; V]$\;        
            $\tilde{Y} = I - \frac{1}{2\real{\sigma}} \cdot (V^\tr B) (V^\tr B)^\tr$;
            $Y = \smb Y & \\ & \tilde{Y} \sme$;
            $R = R + \sqrt{-2 \real{\sigma}} \cdot (\change{E^\tr} V \tilde{Y}^{-1})$\;
            $K = K + (\change{E^\tr} V \tilde{Y}^{-1}) \cdot (V^\tr B)$;
        }
        {
            $Z = [Z \;\; \real{V} \;\; \imag{V}]$\;
            $V_{r} = (\real{V})^\tr B; \quad V_{i} = (\imag{V})^\tr B$\;
            $F_1 = \mb{c} -\real{\sigma} V_{r} - \imag{\sigma} V_{i} \\ \imag{\sigma}V_{r} - \real{\sigma}V_{i} \me; \quad
                F_2 = \mb{c} V_{r} \\ V_{i} \me ; \quad
                F_3 = \mb{c} \imag{\sigma} I_p \\ \real{\sigma}I_p \me$ \;
            $\tilde{Y} = \mb{cc} I_p & \\ & 1/2 I_p \me - \frac{1}{4|\sigma|^2 \real{\sigma}} F_1 F_1^\tr - \frac{1}{4\real{\sigma}} F_2 F_2^\tr -
\frac{1}{2|\sigma|^2} F_3 F_3^\tr$\;
            $Y = \mb{cc} Y & \\ & \tilde{Y} \me$\;
            $R = R + \sqrt{-2\real{\sigma}} \cdot \change{E^\tr}\left( [\real{V} \;\; \imag{V}] \tilde{Y}^{-1} \right) (:, 1:p)$\;
            $K = K + \change{E^\tr}[\real{V} \;\; \imag{V}] \tilde{Y}^{-1} \mb{c} V_{r} \\ V_{i} \me$\;
        }
    }    
    \caption{The RADI iteration, \change{with reduced use of complex arithmetic}.}\label{lrcf:alg:qadi-real}
\end{zAlgorithmFloat}

% {\bf RADI iteration for the generalized Riccati equation.}
\subsection{RADI iteration for the generalized Riccati equation.}
\change{Before we state the final implementation, we shall briefly mention the adaptation of the RADI algorithm for handling generalized Riccati equations}
% .
% The RADI algorithm can be easily adapted for solving the generalized Riccati equation
\begin{equation}
    \label{lrcf:eq:gen-riccati}
    A^\tr XE + E^\tr X A + Q - E^\tr XGXE = 0.
\end{equation}
\change{%
% For simplicity, we only display the variant that uses complex arithmetic.
Multiplying \eqref{lrcf:eq:gen-riccati} by $E^{-\tr}$ from the left and by $E^{-1}$ from the right leads to
\begin{equation}
    \label{lrcf:eq:gen-riccati-reverted}
    (AE^{-1})^\tr X + X (AE^{-1}) + E^{-\tr} C^\tr C E^{-1} - X BB^\tr X = 0.
\end{equation}
The generalized RADI algorithm is then easily derived by running ordinary RADI iterations for the Riccati equation 
\eqref{lrcf:eq:gen-riccati-reverted}, and deploying standard rearrangements to lower the computational expense
(such as solving systems with the matrix $A^\tr + \sigma E^\tr$ instead of $(AE^{-1})^\tr + \sigma I$).}
\change{Algorithm~\ref{lrcf:alg:qadi-real} shows the final implementation, taking into account Proposition~\ref{lrcf:prop:qadi-real} 
and the handling of generalized equations \eqref{lrcf:eq:gen-riccati}.}
% This leads to Algorithm~\ref{lrcf:alg:qadi-generalized}.

% \begin{zAlgorithmFloat}
%     \KwIn{matrices $E, A \in \Real^{n \times n}$, $B \in \Real^{n \times m}$, $C\in \Real^{p \times n}$.}
%     \KwOut{approximation $X \approx Z Y^{-1} Z^\tr$ for the solution of $A^\tr XE + E^\tr XA + C^\tr C - E^\tr X BB^\tr XE =0$.}

%     \vspace*{0.3cm}
%     $R = C^\tr$; $K = 0$; $Y = [\;]$\;
%     \While{$\|R^\tr R\| \geq tol \cdot \|C C^\tr\|$}
%     {
%         Obtain the next shift $\sigma$\;
%         \eIf{first pass through the loop}
%         {
%             $Z = V = \sqrt{-2\real{\sigma}} \cdot (A^\tr + \sigma E^\tr)^{-1} R$\;
%         }
%         {
            
%             $V = \sqrt{-2\real{\sigma}} \cdot (A^\tr - K B^\tr + \sigma E^\tr)^{-1} R$; \tcp{Use SMW if necessary}
%             $Z = [Z \;\; V]$\;        
%         }
        
%         $\tilde{Y} = I - \frac{1}{2\real{\sigma}} \cdot (V^\tr B) (V^\tr B)^\tr$;
%         $Y = \smb Y & \\ & \tilde{Y} \sme$;

%         $R = R + \sqrt{-2 \real{\sigma}} \cdot (E^\tr V \tilde{Y}^{-1})$\;
%         $K = K + (E^\tr V \tilde{Y}^{-1}) \cdot (V^\tr B)$;
%     }    
%     \caption{The RADI iteration for the generalized Riccati equation.}\label{lrcf:alg:qadi-generalized}
% \end{zAlgorithmFloat}

%%%%%%%%%%%%%%%%%%%%%%%%%%%%%%%%%%%%%%%%%%%%%%%%%%%%%%%% 
% .. Shift selection in the qADI algorithm
%%%%%%%%%%%%%%%%%%%%%%%%%%%%%%%%%%%%%%%%%%%%%%%%%%%%%%%%
\subsection{Shift selection}
\label{ssne}

The problem of choosing the shifts in order to accelerate the convergence of the Riccati {ADI} iteration is very similar to the one for
the Lyapunov ADI method. Thus we apply and discuss the techniques presented in \cite{Sab07,Pen00b,BenKS14,BujB14} in the context of the Riccati equation, 
and compare them in several numerical experiments. 
% We mainly focus on strategies where the shifts are generated dynamically during the iteration, but for the sake of completeness we briefly mention
% an approach for precomputed shifts. 
\change{%Because of the close relation of the RADI with the Lyapunov ADI iteration, 
It appears natural to employ the heuristic Penzl
shifts \cite{Pen00b}. There, a small number of approximate eigenvalues of $A$ are 
generated. From this set the values which lead to the smallest magnitude of the 
rational function associated to the ADI iteration are selected in a heuristical 
manner. Simoncini and Lin \cite{LinSim13} have shown that the convergence of the Riccati 
ADI iteration is related to a rational function built from the 
stable eigenvalues of $\nice{H}$. This suggests to carry out the Penzl approach, but to use approximate 
eigenvalues for the Hamiltonian matrix $\nice{H}$ instead of $A$.
Note that, due to the low rank of $Q$ and $G$, we can expect that most of the eigenvalues of $A$ are close to the eigenvalues of $\nice{H}$, see 
the discussion in \cite{BujB14}. Thus in many cases the Penzl shifts generated by $A$ should suffice as well.
Penzl shifts require significant preprocessing computation: in order to approximate the eigenvalues of $M=A$ or $M=\nice{H}$, one
has to build Krylov subspaces with matrices $M$ and $M^{-1}$. All the shifts are computed in this preprocessing stage, and then
simply cycled during the RADI iteration.
Here, we will mainly focus on alternative approaches generating each shift just before it is used. This way we hope to compute a shift which will
better adapt
to the current stage of the algorithm.}
\\

% Hence, 
% $$
%     \max_{i=1, \ldots, n} \prod_{j=1}^k \frac{|\lam^\nice{H}_i - \sigma_j|}{|\lam^\nice{H}_i + \conj{\sigma_j}|}
% $$
% being small, where $\lam^\nice{H}_i$ are now the stable eigenvalues of $\nice{H}$.
% Thus it is reasonable to use Penzl shifts \cite{Pen00b} for the {RADI} iteration as well.
% 
% 
% {\bf Penzl shifts.}
% 
% 
% We first note that, by using the formula $R_k^{(1)}$ of Lemma~\ref{lrcf:tm:residual} and by telescoping the recursive definition of $V_k$,
% we have $\nice{R}(X_k) = R_k^{(1)} (R_k^{(1)})^\tr$, where
% \begin{align*}
%     R_k^{(1)}
%         = &(A^\tr - X_k G - \conj{\sigma_k}I)\left( \prod_{j=2}^k (A^\tr - X_{j-1} G + \sigma_j I)^{-1} (A^\tr - X_{j-1} G - \conj{\sigma_{j-1}}I)
% \right)\times\\
% \times &(A^\tr + \sigma_j I)^{-1} C^\tr
% \end{align*}
% In case of the Lyapunov equation (when $G=0$), taking norms leads to the familiar expression \cite{LiW02,Sab07,Wac13}
% \begin{equation}
%     \label{ssne:eq:residual_lyapunov}
%     \|\nice{R}(X_k)\| \leq \kappa(A)^2 \|C\|^2 \cdot \max_{i=1, \ldots, n} \prod_{j=1}^k \frac{|\lam^A_i - \sigma_j|^2}{|\lam^A_i + \conj{\sigma_j}|^2},
% \end{equation}
% where $\kappa(A)$ is the condition number of the eigenvector matrix in the eigenvalue decomposition of $A$, and $\lam_i^A$ are the eigenvalues of $A$.
% When $G \neq 0$, it is more difficult to obtain an expression such as \eqref{ssne:eq:residual_lyapunov}, since the eigenvector matrix cannot be factored out
% and canceled between the products of the shifted closed loop matrices.
% 
% 
% \\

{\bf Residual Hamiltonian shifts.}

One such approach is motivated by Theorem~\ref{lrcf:tm:residual_equation} and the discussion about shift selection in \cite{BujB14}.
The Hamiltonian matrix
$$
    \tilde{\nice{H}} = \mb{cc} \tilde{A} & G \\ \tilde{C}^\tr\tilde{C} & -\tilde{A}^\tr \me
$$
is associated to the residual equation \eqref{lrcf:eq:residual_equation}, where $\Xi = X_k$ is the approximation after $k$ steps of the RADI iteration.
If $(\lam, \smb r \\ q \sme)$ is a stable eigenpair of $\tilde{\nice{H}}$, and $\sigma_{k+1} = \lam$ is used as the shift, then
$$
    X_k \leq X_{k+1} = X_k - q (q^\tr r)^{-1} q^\tr \leq X.
$$
In order to converge as fast as possible to $X$, it is better to choose such an eigenvalue $\lam$ for which
the update is largest, i.e. the one that maximizes $\|q (q^\tr r)^{-1} q^\tr\|$.
Note that as the {RADI} iteration progresses and the residual matrix $\nice{R}(\Xi)=\tilde{C}^\tr \tilde{C}$ 
converges to zero, the structure of eigenvectors of $\tilde{\nice{H}}$ that belong to its stable eigenvalues is such that
$\|q\|$ becomes smaller and smaller.
Thus, one can further simplify the shift optimality condition, and use the eigenvalue $\lam$ such that the corresponding $q$ has
the largest norm---this is also in line with the discussion in \cite{BujB14}.

However, in practice it is computationally very expensive to determine such an eigenvalue, since the matrix $\tilde{\nice{H}}$ is of order $2n$.
We can approximate its eigenpairs through projection onto some subspace.
If $U$ is an orthonormal basis of the chosen subspace, then
$$
    (U^\tr \tilde{A} U)^\tr \pXtil + \pXtil (U^\tr \tilde{A} U) + (U^\tr \tilde{C}^\tr) (U^\tr \tilde{C}^\tr)^\tr - \pXtil (U^\tr B) (U^\tr B)^\tr \pXtil = 0
$$
is the projected residual Riccati equation with the associated Hamiltonian matrix
\begin{equation}
    \label{exp:eq:projected-residual}
    \pHtil = 
        \mb{cc} 
            U^\tr \tilde{A} U                               & (U^\tr B) (U^\tr B)^\tr \\
            (U^\tr \tilde{C}^\tr) (U^\tr \tilde{C}^\tr)^\tr & -(U^\tr \tilde{A} U)^\tr
        \me.
\end{equation}
Approximate eigenpairs of $\tilde{\nice{H}}$ are $(\hat{\lam}, \smb U \hat{r} \\ U \hat{q} \sme )$, where $(\hat{\lam}, \smb \hat{r} \\ \hat{q} \sme)$
are eigenpairs of $\pHtil$.
Thus, a reasonable choice for the next shift is such $\hat{\lam}$, for which $\|U \hat{q}\| = \|\hat{q}\|$ is the largest.

We still have to define the subspace $\spann{U}$.
One option is to use $V_k$ (or, equivalently, the last $p$ columns of the matrix $Z_k$), which works very well in practice unless $p=1$. 
When $p=1$, all the generated shifts are real, which can make the convergence slow in some cases.
Then it is better to choose the last $\ell$ columns of the matrix $Z_k$; usually already $\ell=2$ or $\ell=5$ or a small multiple of $p$ will suffice. 
% We point out that this shift selection strategy has been adapted to the low-rank Newton-ADI type methods for  Riccati equations in
% \cite{BenHW15}.
An ultimate option is to use the entire $Z_k$, which we denote as $\ell=\infty$. 
This is obviously more computationally demanding, but it provides fast convergence in all cases we tested.
\\

% \todo{This is how the shifts are chosen in the figures below. qADI(H) (resp. qADI(A)) and RKSM(H) (resp. RKSM(A)) use Penzl shifts generated by H (resp. A).
% qADI(self\_1) uses the technique described above with $U=V_k$. qADI(self\_5) uses the same technique with $U=$ max(p, 5) last columns of $Z_k$ (thus the green and brown
% lines coincide when $p\geq 5$). qADI(self\_$\infty$) uses $U=Z_k$. A new shift (or a complex conjugate pair of shifts) is generated in each step. \\
% Also, note that the computation of the new shift requires us to form and compute the eigendecomposition of the projected residual Hamiltonian matrix $\pHtil$.
% This is effectively equivalent to solving the projected residual Riccati equation, and requires the same amount of work as solving the projected ARE in RKSM.
% That is, if we project onto the entire subspace spanned by columns of $Z_k$. If we use $\ell=1, 2, 5, ...$, then the projected Riccati equation
% is far smaller than the one in RKSM.
% }
% \\

% \todo{TODO: Computing $\pHtil$ efficiently, avoiding the computation of the products $\tilde{A}U$, avoiding computing $U$ from scratch for each shift? 
% Zvonimir: Probably the description can be skipped. My code is not doing that efficiently (for now)...Neither for qADI nor RKSM.
% However, this seems to be quite important for the numerical examples.}
% \\

{\bf Residual minimizing shifts.}
The two successive residual factors are connected via the formula
$$
    R_{k+1} = (A^\tr - X_{k+1}G - \conj{\sigma_{k+1}}I) (A^\tr - X_k G + \sigma_{k+1} I)^{-1} R_k.
$$
Our goal is to choose the shifts so that the residual drops to zero as quickly as possible. Locally, once $X_k$ is computed,
this goal is achieved by choosing $\sigma_{k+1}\in\LHP$ so that $\|R_{k+1}\|$ is minimized. This concept was proposed in \cite{BenKS14} for the
low-rank Lyapunov and Sylvester ADI methods. In complete analogy, we define
a rational function $f$ in the variable $\sigma$ by
\begin{equation}
    \label{exp:eq:resmin-function}
    f(\sigma) := \|(A^\tr - X_{k+1}(\sigma) G - \conj{\sigma} I) (A^\tr - X_k G + \sigma I)^{-1} R_k\|^2,
\end{equation}
and wish to find
$$
    \argmin_{\sigma \in \LHP} f(\sigma);
$$
note that $X_{k+1}(\sigma) = X_k + V_{k+1}(\sigma)\tilde{Y}_{k+1}^{-1}(\sigma) V_{k+1}^\tr(\sigma)$ is also a function of $\sigma$.
Since $f$ involves large matrices, we once again project the entire equation to a chosen subspace $U$, and
solve the optimization problem defined by the matrices of the projected problem. 
The optimization problem is solved numerically. Efficient optimization solvers use the gradient of the function $f$; 
after a laborious computation one can obtain an explicit formula for the case $p=1$:
$$
    \nabla f(\sigma_R, \sigma_I)
        = 
            \mb{c}
                \textcolor{white}{-}2\real{R_{k+1}^\tr \cdot \left( (\frac{1}{\sigma_R} I - \tilde{A}^{-1} - \frac{1}{2\sigma_R} (X_{k+1}G\tilde{A}^{-1} + X_{k+1}\tilde{A}^{-\tr}G)) \Delta \right)}  \\
                -2\imag{R_{k+1}^\tr \cdot \left( (\textcolor{white}{\frac{1}{\sigma_R} I}  -\tilde{A}^{-1} - \frac{1}{2\sigma_R} (X_{k+1} G \tilde{A}^{-1} - X_{k+1}\tilde{A}^{-\tr}G)) \Delta \right)}
            \me.
$$
Here $\sigma = \sigma_R + \ii \sigma_I$, $\tilde{A} = A^\tr - X_k G - \sigma I$, $X_{k+1} = X_{k+1}(\sigma)$, $R_{k+1} = R_{k+1}(\sigma)$, and $\Delta = R_{k+1}(\sigma) - R_k$. \\

For $p>1$, a similar formula can be derived, but one should note that the function $f$ is not necessarily differentiable at every point $\sigma$, see,
e.g., \cite{Ove92}.
Thus, a numerically more reliable heuristic is to artificially reduce the problem once again to the case $p=1$. 
This can be done in the following way: let $v$ denote the right singular vector corresponding to the largest singular value of the matrix $R_k$.
Then $R_k\in\Complex^{n \times p}$ in \eqref{exp:eq:resmin-function} is replaced by the vector $R_k v \in \Complex^{n}$.
\\
% The same approach has already been successfully applied to the Lyapunov equation \cite{BenKS14}.

Since numerical optimization algorithms usually require a starting point for the optimization, the two shift generating
approaches may be combined: the residual Hamiltonian shift can be used as the starting point in the optimization for the second approach.
However, from our numerical experience we conclude that the additional computational effort invested in the post-optimization
of the residual Hamiltonian shifts often does not contribute to the convergence. 
\change{The main difficulty is the choice of an adequate subspace $U$ such that the
projected objective function approximates \eqref{exp:eq:resmin-function} well enough. This issue requires futher investigation.}
The rationale is given in the following example.

\begin{example}
    \label{shf:example:heatmap}
    Consider the Riccati equation %``\textsf{CUBE}''
    given in Example 5.2 of \cite{Sim07}: the matrix $A$ is obtained by
    the centered finite difference discretization of the differential equation
    $$
        \change{\partial_t u = \Delta u - 10 x \partial_x u - 1000 y \partial_y u - 10 \partial_z u + b(x, y)f(t),}
    $$
    on a unit cube with 22 nodes in each direction. \change{Thus $A$ is of order $n=10648$;
    the matrices $B \in \Real^{n \times m}$ and $C \in \Real^{p \times n}$ are generated at random, and in this
    example we set $m=p=1$ and $B=C^\tr$.}
    % Note: here B = C' = randn( n, p ).
    % In numerical examples: B = C' = rand( n, p ).

    Suppose that $13$ RADI iterations have already been computed, and that we need to compute a shift to be used in the
    $14$th iteration.
    Let the matrix $U$ contain an orthonormal basis for $Z_{13}$.

    \begin{figure}[t]
        \begin{subfigure}{.48\textwidth}
            \centering
            \includegraphics[width=\textwidth]{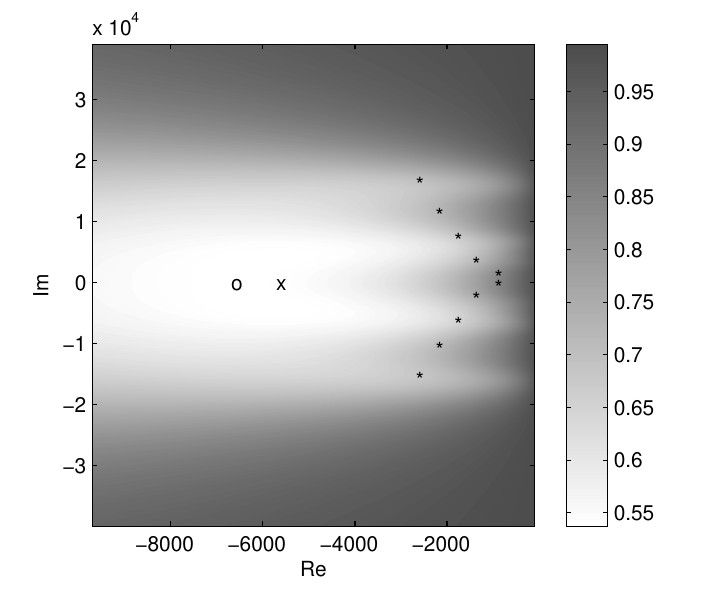}
            \caption{Ratios $\rho^{\mathsf{proj}}(\sigma)=\|R^{\mathsf{proj}}_{14}(\sigma)\| / \|R^{\mathsf{proj}}_{13}\|$ for the projected equation of dimension $13$.}
            \label{shf:fig:heatmap_proj}
        \end{subfigure}
        \hfill
        \begin{subfigure}{.48\textwidth}
            \centering
            \includegraphics[width=\textwidth]{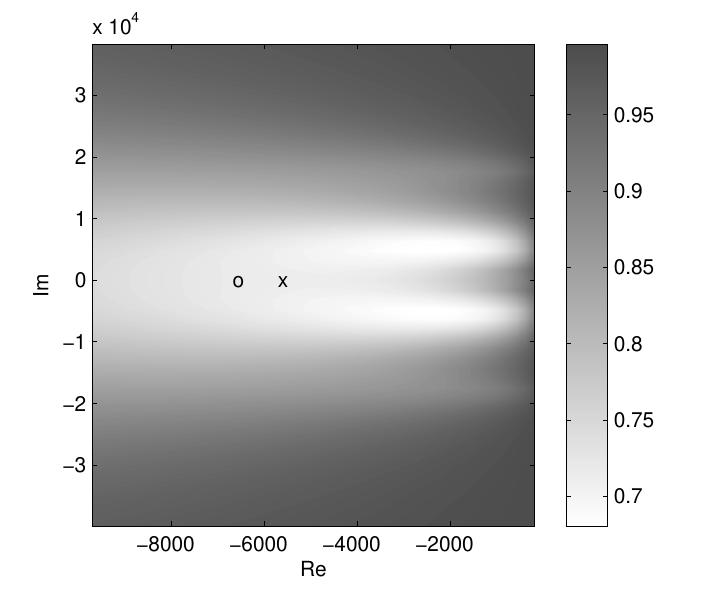}
            \caption{Ratios $\rho(\sigma)=\|R_{14}(\sigma)\| / \|R_{13}\|$ for the original equation of dimension $10648$.}
            \label{shf:fig:heatmap_full}
        \end{subfigure}
        \caption{Each point $\sigma$ of the complex plane is colored according to residual reduction obtained when $\sigma$ is taken as
            the shift in the $14$th iteration of RADI.}
    \end{figure}

    Figure \ref{shf:fig:heatmap_proj} shows a region of the complex plane; stars are at locations of the stable eigenvalues of the 
    projected Hamiltonian matrix \eqref{exp:eq:projected-residual}.
    The one eigenvalue chosen as the residual Hamiltonian shift $\sigma_{\mathsf{ham}}$ is shown as `x'.
    The residual minimizing shift $\sigma_{\mathsf{opt}}$ is shown as `o'.
    Each point $\sigma$ of the complex plane is colored according to the ratio 
    $\rho^{\mathsf{proj}}(\sigma)=\|R^{\mathsf{proj}}_{14}(\sigma)\| / \|R^{\mathsf{proj}}_{13}\|$,
    where $R^{\mathsf{proj}}$ is the residual for the {\it projected} Riccati equation.
    The ratio 
    $\rho^{\mathsf{proj}}(\sigma_{\mathsf{ham}}) \approx 0.54297$ is not far from the optimal ratio 
    $\rho^{\mathsf{proj}}(\sigma_{\mathsf{opt}}) \approx 0.53926$.

    On the other hand, Figure \ref{shf:fig:heatmap_full} shows the complex plane colored according to ratios 
    for the original system of order $10648$,
    $\rho(\sigma) = \|R_{14}(\sigma)\| / \|R_{13}\|$.
    Neither of the values 
    $\rho(\sigma_{\mathsf{ham}}) \approx 0.71510$ and 
    $\rho(\sigma_{\mathsf{opt}}) \approx 0.71981$ is optimal, but they both offer a reasonable reduction
    of the residual norm in the next step.
    In this case, $\sigma_{\mathsf{ham}}$ turns out even to give a slightly better residual reduction for the original equation
    than $\sigma_{\mathsf{opt}}$, making the extra effort in running the numerical optimization algorithm
    futile.
\end{example}

%!TEX root = qADI_LR.tex

%%%%%%%%%%%%%%%%%%%%%%%%%%%%%%%%%%%%%%%%%%%%%%%%%%%%%%%% 
% .. Numerical experiments
%%%%%%%%%%%%%%%%%%%%%%%%%%%%%%%%%%%%%%%%%%%%%%%%%%%%%%%%
\section{Numerical experiments}
\label{num}

\newcommand{\specialcell}[2][c]{%
  \begin{tabular}[#1]{@{}c@{}}#2\end{tabular}}

In this section we show a number of numerical examples, with several objectives in mind.
First, our goal is to compare different low-rank implementations of the Riccati ADI algorithm mentioned in this paper: the low-rank qADI
proposed in \cite{WonB05,morWonB07}, the Cayley transformed subspace iteration \cite{LinSim13,MasOR14}, and the complex and real variants of the RADI iteration
\eqref{lrcf:eq:qadi-lr}.
Second, we compare performance of the RADI approach against other methods for solving large-scale Riccati equations, namely the
rational Krylov subspace method (RKSM) \cite{SimSM13a}, the extended block Arnoldi (EBA) method \cite{Sim07,HeyJ09},
and the Newton-ADI algorithm \cite{BenLP08,BenS10,BenHW15}.

Finally, we discuss various shift strategies for the RADI iteration described in the previous section.
\\

The numerical experiments are run on a desktop computer with a four-core Intel Core i5-4690K processor and 16GB RAM.
All algorithms and testing routines are implemented and executed in MATLAB R2014a, running on Microsoft Windows~8.1.

\begin{example}
    \label{num:ex:CUBE}
    Consider again the Riccati benchmark \bench{CUBE} from Example \ref{shf:example:heatmap}.
   \change{We use three versions of this example: the previous setting with $n=10648$, $m=p=1$ and $m=p=10$, 
    and later on a finer discretization with $n=74088$ and $m=10$, $p=1$.}

    \change{Table~\ref{ex:tbl:diffimplem} collects timings in seconds for four different low-rank implementations of the Riccati ADI algorithm. The table shows
only the time needed to run $80$ iteration steps;}
    time spent for computation of the shifts used by all four variants is not included (in this case, $20$ precomputed Penzl shifts were used). 
    All four variants compute exactly the same iterates, as we have proved in Theorem \ref{equiv:tm}.
\begin{table}[h]
\centering
 \begin{tabular}{|c|r|r|}
            \hline
            Implementation                                    & time, $m,p = 1$ & time, $m,p = 10$ \\
            \hline \hline
            Wong and Balakrishnan \cite{WonB05,morWonB07}     & $127.61$      & $750.89$ \\
            Cayley subspace iteration \cite{LinSim13,MasOR14} & $21.67$       & $167.02$ \\
            RADI -- Algorithm~\ref{lrcf:alg:qadi-complex}     & $21.51$       & $51.92$ \\
            RADI -- Algorithm~\ref{lrcf:alg:qadi-real}        & $11.14$       & $26.35$ \\ \hline
        \end{tabular}
\caption{Results obtained with different implementations for \bench{CUBE} with $n=10648$.}
\label{ex:tbl:diffimplem}        
\end{table}

    Clearly, the real variant of iteration \eqref{lrcf:eq:qadi-lr}, implemented as in Algorithm~\ref{lrcf:alg:qadi-real},
    outperforms all the others.\footnote{Note that the generated Penzl shifts come in complex conjugate pairs.}
    Thus we use this implementation in the remaining numerical experiments.
    The RADI algorithms mostly obtain the advantage over the Cayley subspace iteration \change{because of the cheap computation of the residual norm.  In the
latter algorithm, costly orthogonalization procedures are required for this task, and after some point 
these compensate the computational gains from the
easier linear systems (cf. Section~\ref{ssec:linsysRADI}).}
%     not having to
%     perform any orthogonalization, which the latter algorithm needs in order to compute the residuals.
%     This advantage is more emphasized for larger values of $p$.
    Also, the times for the algorithm of Wong and Balakrishnan shown in the table do not include the (very costly) 
    computation of the residuals at all, so their actual execution times are even higher.

    Next, we compare various shift strategies for RADI, as well as EBA, RKSM, and Newton-ADI algorithms.
    For RADI, we have the following strategies:
    \begin{itemize}
        \item $20$ precomputed Penzl shifts (``RADI -- Penzl'') generated by using the Krylov subspaces of dimensions $40$ with matrices $A$ and $A^{-1}$;
        \item residual Hamiltonian shifts (``RADI -- Ham''), with $\ell=2p$, $\ell=6p$, and $\ell=\infty$;
        \item residual minimizing shifts (``RADI -- Ham+Opt''), with $\ell=2p$, $\ell=6p$, and $\ell=\infty$.
    \end{itemize}
    For the RKSM, we have implemented the algorithm so that the use of complex arithmetic is minimized by merging two
consecutive
    steps with complex conjugate shifts \cite{Ruh94c}. \change{We use the adaptive shift strategy, 
    implemented as described in
\cite{DruS11}. EBA and RKSM require the solution of a projected CARE which can become expensive if $p>1$. 
Hence, this small scale solution is carried out
only periodically in every 5th or every 10th step---in the results, we display the variant that was faster.}
%     \begin{itemize}
%         \item Penzl shifts, computed in the same way as for the RADI iteration;
%         \item 
%     \end{itemize}

    For all methods, the threshold for declaring convergence is reached once the relative residual is less than $tol=10^{-11}$.
    A summary of the results for all the different methods and strategies is shown in Table \ref{ex:tbl:results}.
    The column ``final subspace dimension'' displays the number of columns of the matrix $Z$, where $X\approx ZZ^\tr$ is
    the final computed approximation. Dividing this number by $p$ (for EBA, by $2p$), we obtain the number of iterations used in a particular 
    method.
    Just for the sake of completeness, we have also included a variant of the Newton-ADI algorithm \cite{BenLP08} with Galerkin projection
\cite{BenS10}. Without the Galerkin projection, the Newton-ADI algorithm could not compete with the other methods.
    % This variant runs the inner Lyapunov-ADI iterations until their relative residual becomes smaller than $tol$.
    The recent developments from \cite{BenHW15}, which make the Newton-ADI algorithm more competitive, are beyond the scope of this study.

    \change{It is interesting to analyze the timing breakdown for RADI and RKSM methods. These timings are listed in Table~\ref{ex:tbl:breakdown} for the
    \bench{CUBE} example with $m=p=10$ where a significant amount of time is spent for tasks other than solving linear systems.}

\begin{table}[h]
\centering
    \begin{tabular}{|c|r|r|} 
        \hline
        method & subtask & time, $m=p=10$ \\
        \hline \hline
        \multirow{3}{*}{\specialcell{RADI -- Penzl: \\ 135 iterations}} 
            & precompute shifts          & 5.31  \\
            & solve linear systems       & 43.24 \\
            \cline{2-3}
            & total                      & 49.19 \\
        \hline \hline
        \multirow{4}{*}[.5em]{\specialcell{RADI -- Ham, $\ell=2p$: \\ 139 iterations}} 
            & solve linear systems       & 46.42 \\
            & compute shifts dynamically & 1.76 \\
            \cline{2-3}
            & total                      & 48.79 \\
        \hline \hline
        \multirow{4}{*}[.5em]{\specialcell{RADI -- Ham, $\ell=6p$: \\ 100 iterations}} 
            & solve linear systems       & 32.73 \\
            & compute shifts dynamically & 2.75 \\
            \cline{2-3}
            & total                      & 35.92 \\
        \hline \hline
        \multirow{4}{*}[.5em]{\specialcell{RADI -- Ham, $\ell=\infty$: \\ 74 iterations}} 
            & solve linear systems       & 24.74 \\
            & compute shifts dynamically & 32.44 \\
            \cline{2-3}
            & total                      & 57.51 \\
        \hline \hline
        \multirow{4}{*}[-.7em]{\specialcell{RKSM -- adaptive: \\ 79 iterations}} 
        % This is for the "new" version of RKSM -- adaptive, which projects on every 10th step.
            & solve linear systems       & \change{18.45} \\
            & orthogonalization          & \change{4.14} \\
            & compute shifts dynamically & \change{12.02} \\
            & solve projected equations  & \change{15.98} \\
            \cline{2-3}
            & total                      & \change{53.82} \\
        \hline
    \end{tabular}
    \caption{Times spend in different subtasks in the RADI iteration and RKSM for \bench{CUBE}
 with $m=p=10$.}
\label{ex:tbl:breakdown}
\end{table}

    As $\ell$ increases, the cost of computing shifts in RADI increases as well---the projection subspace 
    gets larger, and more effort is needed to orthogonalize its basis and compute the eigenvalue decomposition of the
    projected Hamiltonian matrix. This effort is, in the \bench{CUBE} benchmark, awarded by a decrease in the number of iterations.
    However, there is a trade-off here: the extra computation does outweigh the saving in the number of iterations for 
    sufficiently large $\ell$.
    Convergence history for \bench{CUBE} is plotted in Figure \ref{num:fig:CUBE}; to reduce the clutter, only the selected few methods are shown.

    \change{The fact that in each step RADI solves linear systems with $p+m$ right hand side vectors, compared to only $p$
    vectors in RKSM, may become noticable when $m$ is larger than $p$. This effect is shown in Table \ref{ex:tbl:results} for 
    \bench{CUBE} with $m=10$ and $p=1$. Unlike these two methods, EBA can precompute the LU factorization of $A$, and win by a large
    margin in this test case.}

    \begin{figure}[t]
        \begin{subfigure}{.5\textwidth}
            \centering
            \includegraphics[width=\textwidth]{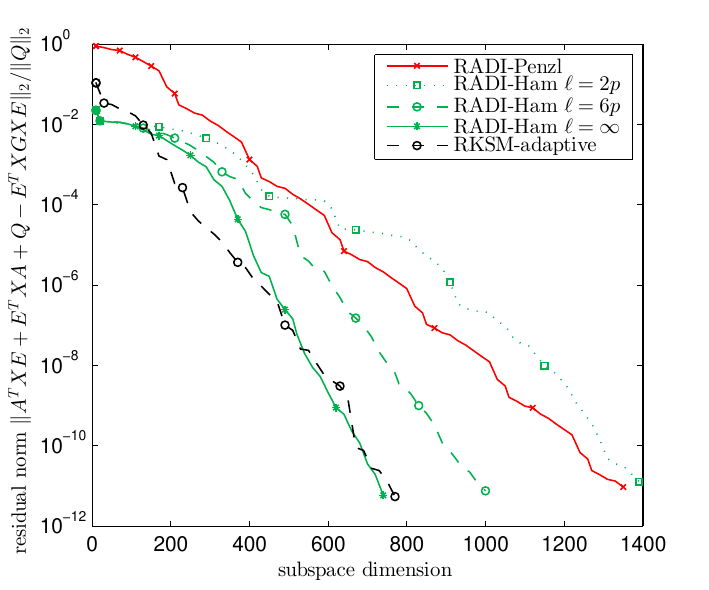}
            \caption{Relative residual versus the subspace dimension used by an algorithm.}
            \label{num:fig:CUBE-res}
        \end{subfigure}
        \hfill
        \begin{subfigure}{.5\textwidth}
            \centering
            \includegraphics[width=\textwidth]{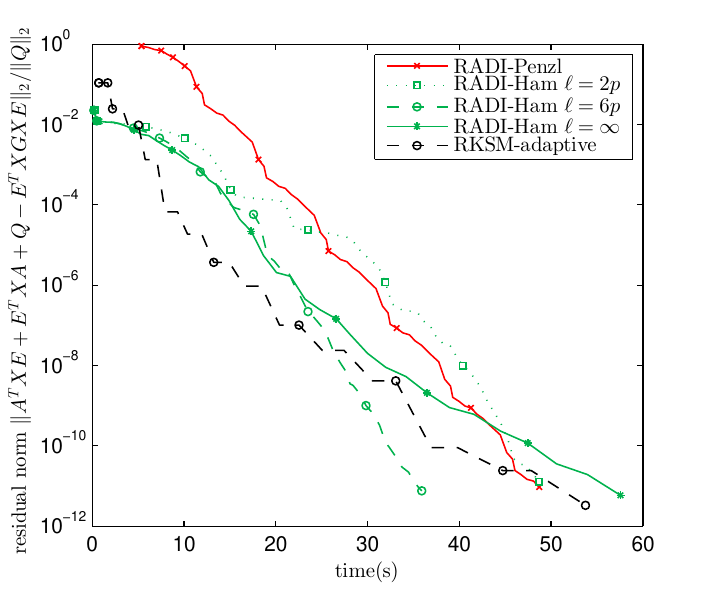}
            \caption{Relative residual versus time.}
            \label{num:fig:CUBE-mil}
        \end{subfigure}
        \caption{Algorithm performances for benchmark \bench{CUBE} ($n=10648, m=p=10$)}
        \label{num:fig:CUBE}
    \end{figure}
\end{example}

\begin{table}
{\footnotesize \vspace*{-1cm}
\begin{tabular}{|c|r|r|r|r|} 
    \hline
    example & method & no.~iterations & final subspace dim. & time \\
    \hline \hline
    \multirow{10}{*}{\rot{\specialcell{\bench{CUBE} \\ $n=10648$ \\ $m=p=1$}}}
          & RADI -- Penzl                  & $ 97$                  & $97 $          & $18.96$ \\
          & RADI -- Ham, $\ell=2p$         & $119$                  & $119$          & $17.10$ \\
          & RADI -- Ham, $\ell=6p$         & $ 99$                  & $99 $          & $14.15$ \\
          & RADI -- Ham, $\ell=\infty$     & $ 75$                  & $\mathbf{75}$  & $11.60$ \\
          & RADI -- Ham+Opt, $\ell=2p$     & $122$                  & $122$          & $17.87$ \\
          & RADI -- Ham+Opt, $\ell=6p$     & $103$                  & $103$          & $16.70$ \\
          & RADI -- Ham+Opt, $\ell=\infty$ & $108$                  & $108$          & $18.44$ \\
          & RKSM -- adaptive               & \change{$ 83$}         & \change{$83$}  & \change{$14.80$} \\
          & EBA                            & \change{$111$}         & \change{$222$} & \change{$\mathbf{6.23}$} \\  %%% EBA with projection in every 10th step + A'=LU
          & Newton-ADI                     & $2$ outer, $296$ inner & $192$          & $42.11$ \\
    \hline
    \multirow{10}{*}{\rot{\specialcell{\bench{CUBE} \\ $n=10648$ \\ $m=p=10$}}}
          & RADI -- Penzl                  & $135$                  & $1350$          & $49.19 $ \\
          & RADI -- Ham, $\ell=2p$         & $139$                  & $1390$          & $48.79 $ \\
          & RADI -- Ham, $\ell=6p$         & $100$                  & $1000$          & $35.92 $ \\
          & RADI -- Ham, $\ell=\infty$     & $ 74$                  & $\mathbf{740}$  & $57.51 $ \\
          & RADI -- Ham+Opt, $\ell=2p$     & $ 87$                  & $870 $          & $\mathbf{30.59}$ \\
          & RADI -- Ham+Opt, $\ell=6p$     & $ 90$                  & $900 $          & $33.20 $ \\
          & RADI -- Ham+Opt, $\ell=\infty$ & $ 90$                  & $900 $          & $119.55$ \\
          & RKSM -- adaptive               & \change{$79$}          & \change{$790$}  & \change{$53.82$} \\  %%% RKSM with projection in every 5th step
          & EBA                            & \change{$91$}          & \change{$1820$} & \change{$230.57$} \\  %%% EBA with projection in every 10th step + A'=LU
          & Newton-ADI                     & $2$ outer, $202$ inner & $1960$          & $75.60$ \\
    \hline
    \multirow{10}{*}{\rot{\specialcell{\bench{CUBE} \\ $n=74088$ \\ $m=10, \; p=1$}}}
            & RADI -- Penzl                  & $139$                  & $139$         & $1048.60$ \\
            & RADI -- Ham, $\ell=2p$         & $97$                   & $97$          & $617.62$ \\
            & RADI -- Ham, $\ell=6p$         & $81$                   & $81$          & $506.37$ \\
            & RADI -- Ham, $\ell=\infty$     & $72$                   & $72$          & $446.64$ \\
            & RADI -- Ham+Opt, $\ell=2p$     & $101$                  & $101$         & $621.38$ \\
            & RADI -- Ham+Opt, $\ell=6p$     & $93$                   & $93$          & $571.34$ \\
            & RADI -- Ham+Opt, $\ell=\infty$ & $63$                   & $\mathbf{63}$ & $387.43$ \\
            & RKSM -- adaptive               & $73$                   & $73$          & $338.78$ \\  %%% RKSM with projection in every 5th step
            & EBA                            & $81$                   & $162$          & $\mathbf{30.45}$ \\  %%% EBA with projection in every 10th step + A'=LU
            & Newton-ADI                     & $2$ outer, $288$ inner & $968$         & $1546.29$ \\
    \hline
    \multirow{10}{*}{\rot{\specialcell{\bench{CHIP} \\ $n=20082$ \\ $m=1, \; p=5$}}}
          & RADI -- Penzl                  & $33$                  & $165$          & $51.57$ \\
          & RADI -- Ham, $\ell=2p$         & $36$                  & $180$          & $30.32$ \\
          & RADI -- Ham, $\ell=6p$         & $29$                  & $145$          & $24.36$ \\
          & RADI -- Ham, $\ell=\infty$     & $26$                  & $130$          & $22.64$ \\
          & RADI -- Ham+Opt, $\ell=2p$     & $29$                  & $145$          & $23.97$ \\
          & RADI -- Ham+Opt, $\ell=6p$     & $26$                  & $130$          & $22.26$ \\
          & RADI -- Ham+Opt, $\ell=\infty$ & $25$                  & $\mathbf{125}$ & $22.33$ \\
          & RKSM -- adaptive               & \change{$26$}         & \change{$130$} & \change{$23.33$} \\  %%% RKSM with projection in every 5th step
          & EBA                            & \change{$26$}         & \change{$260$} & \change{$\mathbf{6.69}$} \\  %%% EBA with projection in every 5th step + A'=LU
          & Newton-ADI                     & $2$ outer, $64$ inner & $204$          & $54.04$\\
    \hline
    \multirow{10}{*}{\rot{\specialcell{\bench{IFISS} \\ $n=66049$, \\ $m=p=5$}}}
          & RADI -- Penzl                  & $>50$                 & $>250$           & \\
          & RADI -- Ham, $\ell=2p$         & $ 22$                 & $110 $           & $17.21$ \\
          & RADI -- Ham, $\ell=6p$         & $ 19$                 & $\mathbf{95}$    & $15.37$ \\
          & RADI -- Ham, $\ell=\infty$     & $ 20$                 & $100 $           & $17.46$ \\
          & RADI -- Ham+Opt, $\ell=2p$     & $ 27$                 & $135 $           & $21.12$ \\
          & RADI -- Ham+Opt, $\ell=6p$     &                       & did not converge & \\
          & RADI -- Ham+Opt, $\ell=\infty$ &                       & did not converge & \\
          & RKSM -- adaptive               & $26$                  & $130$            & $22.28$ \\  %%% RKSM with projection in every step
          & EBA                            & $11$                  & $110$            & \change{$\mathbf{9.26}$} \\  %%% EBA with projection in every 10th step + A'=LU
          & Newton-ADI                     & $2$ outer, $46$ inner & $250$            & $38.05$ \\
    \hline
    \multirow{10}{*}{\rot{\specialcell{\bench{RAIL} \\ $n=317377$, \\  $m=7, \; p=6$}}}
          & RADI -- Penzl                  & $66$                  & $396$          & $182.60$ \\
          & RADI -- Ham, $\ell=2p$         & $49$                  & $294$          & $131.34$ \\
          & RADI -- Ham, $\ell=6p$         & $43$                  & $258$          & $127.11$ \\
          & RADI -- Ham, $\ell=\infty$     & $46$                  & $276$          & $197.06$ \\
          & RADI -- Ham+Opt, $\ell=2p$     & $46$                  & $276$          & $124.13$ \\
          & RADI -- Ham+Opt, $\ell=6p$     & $40$                  & $240$          & $\mathbf{120.04}$ \\
          & RADI -- Ham+Opt, $\ell=\infty$ & $39$                  & $\mathbf{234}$ & $158.89$ \\
          & RKSM -- adaptive               & $41$                  & $246$          & $188.60$ \\   %%% RKSM with projection in every 10th step
          & EBA                            & $91$                  & $1092$         & $916.21$ \\  %%% EBA with projection in every 10th step + A'=LU
          & Newton-ADI                     & $1$ outer, $62$ inner & $372$          & $279.90$ \\
    \hline
    \multirow{10}{*}{\rot{\specialcell{\bench{LUNG} \\ $n=109460,$ \\ $m=p=10$}}}
            & RADI -- Penzl                  &      & did not converge & \\
            & RADI -- Ham, $\ell=2p$         & $31$ & $310$            & $30.03$ \\
            & RADI -- Ham, $\ell=6p$         & $28$ & $280$            & $30.22$ \\
            & RADI -- Ham, $\ell=\infty$     & $26$ & $260$            & $34.83$ \\
            & RADI -- Ham+Opt, $\ell=2p$     & $25$ & $250$            & $22.33$ \\
            & RADI -- Ham+Opt, $\ell=6p$     & $17$ & $\mathbf{170}$   & $\mathbf{17.74}$ \\
            & RADI -- Ham+Opt, $\ell=\infty$ & $17$ & $\mathbf{170}$   & $19.02$ \\
            & RKSM -- adaptive               & $61$ & $610$            & $114.22$ \\  %%% RKSM with projection in every 5th step
            & EBA                            &      & did not converge & \\  %%% EBA with projection in every 10th step + A'=LU
            & Newton-ADI                     &      & did not converge & \\
        \hline
\end{tabular}
\caption{Results of the numerical experiments.}
\label{ex:tbl:results}
}
\end{table}

\begin{example}
    \label{num:ex:CHIP}\label{num:ex:FLOW}
    Next, we run the Riccati solvers for the well-known benchmark example 
    \bench{CHIP}.
    All coefficient matrices for the Riccati equation are taken as they are found in the 
    Oberwolfach Model Reduction Benchmark Collection \cite{Obe05}.
    Here we solve the generalized Riccati equation \eqref{lrcf:eq:gen-riccati}.

    The cost of precomputing shifts is very high in case of \bench{CHIP}.
    One fact not shown in the table is that all algorithms which compute shifts dynamically have already solved the Riccati equation
    before \change{``RADI -- Penzl'' has even started}.

    % \bench{FLOW} shows the benefit of efficient dynamical shift computation. While all methods (apart from ``RADI -- Penzl'')
    % need approximately the same number of iterations, those that can compute shifts quickly (``RADI -- Ham, $\ell=2p, 6p$'' 
    % and ``RADI -- Ham+Opt, $\ell=2p$'') are the fastest.
    % On the other hand, the version with residual minimizing shifts does not converge for $\ell=6p, \infty$: it quickly reaches the relative residual of
    % about $10^{-10}$, and then gets stuck by continually using shifts very close to zero.
    % We have not observed such behavior with residual Hamiltonian shifts.
\end{example}

\begin{example}
    \label{num:ex:IFISS}
    We use the \bench{IFISS} 3.2. finite-element package \cite{ifiss} to generate the coefficient matrices 
    for a generalized Riccati equation.
    We choose the provided example T-CD 2 which represents a finite element discretization of a two-dimensional 
    convection diffusion equation on a square domain.
    The leading dimension is $n=66049$, with $E$ symmetric positive definite, and $A$ non-symmetric.
    The matrix $B$ consists of $m=5$ randomly generated columns, and $C=[C_1, \; 0]$ with random $C_1\in\Real^{5 \times 5}$ ($p=5$).

    In this example, the RADI iteration with Penzl shifts converges very slowly. The
    RADI iteration with dynamically generated shifts are quite fast, and the final subspace dimension is
    smallest among all methods.
    On the other hand, the version with residual minimizing shifts does not converge for $\ell=6p, \infty$: 
    it quickly reaches the relative residual of
    about $10^{-7}$, and then gets stuck by continually using shifts very close to zero.
    % We have not observed such behavior with residual Hamiltonian shifts.
    % Similarly as in \bench{FLOW}, there is a problem with residual minimizing shifts with $\ell=6p, \infty$, which
    % are able to reach relative residual of about $10^{-7}$, and then stagnate.
\end{example}

\begin{example}
    \label{num:ex:rail}\change{
    The example \bench{RAIL} is a larger version of the steel profile cooling model from the Oberwolfach Model Reduction Benchmark Collection
\cite{Obe05}. A finer finite element discretization was used for the heat equation resulting in
   a generalized CARE $n=317377$, $m=7$, $p=6$, and $E$ and $A$ symmetric positive and negative definite, respectively.
   Once again, there is a trade-off between (questionably) better shifts with larger $\ell$ and faster computation with lower $\ell$. }
\end{example}

\begin{example}
    \label{num:ex:lung}\change{
    The final example \bench{LUNG} from the UF Sparse Matrix Collection models temperature and water vapor transport in the human lung.
    It provides matrices with leading dimension $n=109460$, $E=I$, $A$ nonsymmetric, and $B,C$ are generated as random matrices with 
    $m=p=10$.
    This example shows the importance of proper shift generation: precomputed shifts are completely useless, while dynamically generated
    ones show different rates of success.
    % Only the RADI iteration appears to be able to handle this CARE for appropriately generated shift parameters. 
    The projection based methods (RKSM, EBA) 
    encountered problems at the numerical solution of the projected ARE. Either the complete algorithm broke down or convergence speed was reduced. 
    Similar issues were encountered at the Galerkin acceleration stage in the Newton-ADI method.}
\end{example}

Let us summarize the findings from these and a number of other numerical examples we used to test the algorithms.
% The RADI algorithm using Penzl shifts tends to perform worse than the RKSM using the same shifts, and in general needs most iterations of
% all the variants listed above.
\change{Clearly, using dynamically generated shifts for the RADI iteration has many benefits compared to the precomputed Penzl shifts.
Not only that the number of iterations and running time are reduced, but the convergence is more reliable.
% Furthermore, the total amount of time needed to run a fixed number of RADI iteration steps is far lower
% when the shifts are not precomputed---the accumulated time to generate shifts dynamically is lower than the time needed to build
% Krylov subspaces for the Penzl shifts.
Further, there frequently exists a small value of $\ell$ for which one or both of the dynamical shift strategies converge 
in a number of iterations comparable to runs with $\ell=\infty$, and in far less time.
However, an a-priori method of determining a sufficiently small $\ell$ with such properties is still to be found, and a topic of our future research.
}
\change{RADI appears to be quite competitive with other state of the art algorithms for solving large scale CAREs.
It frequently generates solutions using the lowest dimensional subspace.
Since RADI iterations do not require any orthogonalization nor solving projected CAREs,
the algorithm may outperform RKSM and EBA in problems where the final subspace dimension is high.
On the other hand, the later methods may have advantage when running time is
dominated by solving linear systems.
It seems that for now, there is no single algorithm of choice that would consistently and reliably run fastest.
}

\begin{figure}[t]
    \begin{subfigure}{.5\textwidth}
        \centering
        \includegraphics[width=\textwidth]{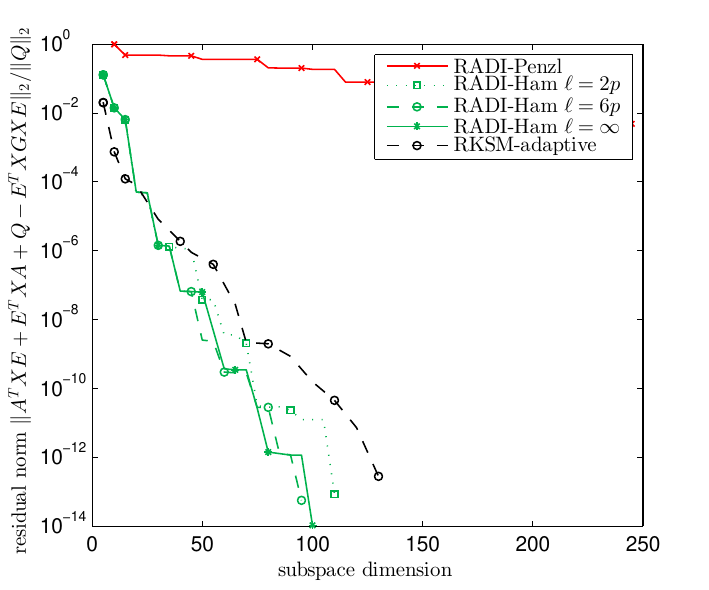}
        \caption{Relative residual versus the subspace dimension used by an algorithm.}
        \label{num:fig:IFISS-res}
    \end{subfigure}
    \hfill
    \begin{subfigure}{.5\textwidth}
        \centering
        \includegraphics[width=\textwidth]{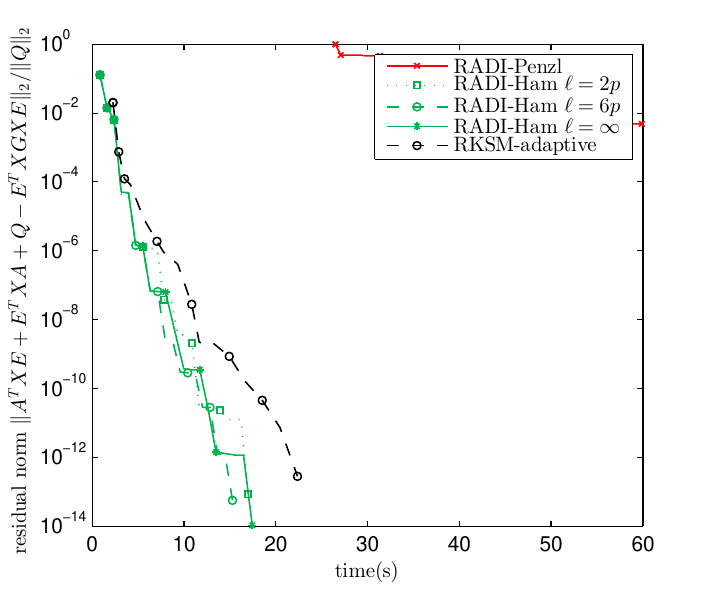}
        \caption{Relative residual versus time.}
        \label{num:fig:IFISS-mil}
    \end{subfigure}

    \caption{Algorithm performances for benchmark \bench{IFISS} ($n=66049, m=p=5$)}
    \label{num:fig:IFISS}
\end{figure}

%!TEX root = qADI_LR.tex

%%%%%%%%%%%%%%%%%%%%%%%%%%%%%%%%%%%%%%%%%%%%%%%%%%%%%%%% 
% .. Conclusion
%%%%%%%%%%%%%%%%%%%%%%%%%%%%%%%%%%%%%%%%%%%%%%%%%%%%%%%%
\section{Conclusion}
\label{con}

In this paper, we have presented a new low-rank RADI algorithm for computing solutions of large scale Riccati equations.
We have shown that this algorithm produces exactly the same iterates as three previously known methods
(for which we suggest the common name ``Riccati ADI methods''), 
but it does so in a computationally far more efficient way.
As with other Riccati solvers, the performance is heavily dependent on the choice of shift parameters.
We have suggested several strategies on how this may be done; some of them show very promising results, making
the RADI algorithm competitive with the fastest large scale Riccati solvers.

\section{Acknowledgements}
    Part of this work was done while the second author
    was a postdoctoral researcher at Max Planck Institute Magdeburg, Germany.
    The second author would also like to acknowledge the support of the Croatian
    Science Foundation under grant HRZZ-9345.

\bibliographystyle{plain}
\bibliography{qADI_LR}

\begin{thebibliography}{10}

\bibitem{AmoB10}
L.~Amodei and J.-M. Buchot.
\newblock {An invariant subspace method for large-scale algebraic {R}iccati
  equation}.
\newblock {\em {Appl. Numer. Math.}}, 60(11):1067--1082, 2010.

\bibitem{BujB14}
P.~Benner and Z.~Bujanovi{\'c}.
\newblock {On the solution of large-scale algebraic {R}iccati equations by
  using low-dimensional invariant subspaces}.
\newblock {\em Linear Algebra and Its Applications}, (488):430--459, 2016.

\bibitem{morBenKS13}
P.~Benner, P.~K{\"u}rschner, and J.~Saak.
\newblock {An improved numerical method for balanced truncation for symmetric
  second order systems}.
\newblock {\em {Math. Comput. Model. Dyn. Sys.}}, 19(6):593--615, 2013.

\bibitem{BenKS13}
P.~Benner, P.~K{\"u}rschner, and J.~Saak.
\newblock {Efficient handling of complex shift parameters in the low-rank
  {C}holesky factor {ADI} method}.
\newblock {\em Numerical Algorithms}, 62(2):225--251, 2013.

\bibitem{BenKS14}
P.~Benner, P.~K{\"u}rschner, and J.~Saak.
\newblock {Self-generating and efficient shift parameters in {ADI} methods for
  large {L}yapunov and {S}ylvester equations}.
\newblock {\em {Electr. Trans. Num. Anal.}}, 43:142--162, 2014.

\bibitem{BenLP08}
P.~Benner, J.-R. Li, and T.~Penzl.
\newblock {Numerical solution of large {L}yapunov equations, {R}iccati
  equations, and linear-quadratic control problems}.
\newblock {\em {Numer. Lin. Alg. Appl.}}, 15(9):755--777, 2008.

\bibitem{BenLT09}
P.~Benner, R.-C. Li, and N.~Truhar.
\newblock {On the {ADI} method for {S}ylvester equations}.
\newblock {\em J. Comput. Appl. Math.}, 233(4):1035--1045, december 2009.

\bibitem{BenS10}
P.~Benner and J.~Saak.
\newblock {A Galerkin-Newton-ADI method for solving large-scale algebraic
  Riccati equations}.
\newblock Preprint SPP1253-090, SPP1253, January 2010.
\newblock
  \url{http://www.am.uni-erlangen.de/home/spp1253/wiki/index.php/Preprints}.

\bibitem{BenS13}
P.~Benner and J.~Saak.
\newblock {Numerical solution of large and sparse continuous time algebraic
  matrix {R}iccati and {L}yapunov equations: a state of the art survey}.
\newblock {\em GAMM Mitteilungen}, 36(1):32--52, August 2013.

\bibitem{Ben15}
Peter Benner.
\newblock {Theory and Numerical Solution of Differential and Algebraic Riccati
  Equations}.
\newblock In Peter Benner, Matthias Bollh{\"o}fer, Daniel Kressner, Christian
  Mehl, and Tatjana Stykel, editors, {\em {Numerical Algebra, Matrix Theory,
  Differential-Algebraic Equations and Control Theory}}, pages 67--105.
  Springer International Publishing.

\bibitem{Bin14}
D.~Bini, B.~Iannazzo, and B.~Meini.
\newblock {\em {Numerical Solution of Algebraic {R}iccati Equations}}.
\newblock {Fundamentals of Algorithms}. {SIAM}, 2012.

\bibitem{DruS11}
V.~Druskin and V.~Simoncini.
\newblock {Adaptive rational {K}rylov subspaces for large-scale dynamical
  systems}.
\newblock {\em Systems and Control Letters}, 60(8):546--560, 2011.

\bibitem{GolV13}
G.~H. Golub and C.~F. {Van~Loan}.
\newblock {\em {Matrix Computations}}.
\newblock Johns Hopkins University Press, Baltimore, fourth edition, 2013.

\bibitem{BenHW15}
M.~Heinkenschloss, H.~K. Weichelt, P.~Benner, and J.~Saak.
\newblock {An inexact low-rank Newton--ADI method for large-scale algebraic
  Riccati equations}.
\newblock {\em {Appl. Numer. Math.}}, 108:125--142, 2016.

\bibitem{HeyJ09}
M.~Heyouni and K.~Jbilou.
\newblock {An extended block {A}rnoldi algorithm for large-scale solutions of
  the continuous-time algebraic {R}iccati equation}.
\newblock {\em {Electr. Trans. Num. Anal.}}, 33:53--62, 2009.

\bibitem{Obe05}
J.~G. Korvink and E.~B. Rudnyi.
\newblock {Oberwolfach Benchmark Collection}.
\newblock In P.~Benner, D.~C. Sorensen, and V.~Mehrmann, editors, {\em
  {Dimension Reduction of Large-Scale Systems}}, volume~45 of {\em {Lecture
  Notes in Computational Science and Engineering}}, pages 311--315. Springer
  Berlin Heidelberg, 2005.

\bibitem{LanR95}
P.~Lancaster and L.~Rodman.
\newblock {\em {The Algebraic {R}iccati Equation}}.
\newblock Oxford University Press, Oxford, 1995.

\bibitem{Lau79}
A.~J. Laub.
\newblock {A {S}chur method for solving algebraic {R}iccati equations}.
\newblock {\em {{IEEE} Trans. Automat. Control}}, AC-24:913--921, 1979.

\bibitem{LevR93}
N.~Levenberg and L.~Reichel.
\newblock {A generalized {ADI} iterative method}.
\newblock {\em Numer. Math.}, 66(1):215--233, 1993.

\bibitem{LiW02}
J.-R. Li and J.~White.
\newblock {Low rank solution of {L}yapunov equations}.
\newblock {\em {{SIAM} J. Matrix Anal. Appl.}}, 24(1):260--280, 2002.

\bibitem{LinSim13}
Y.~Lin and V.~Simoncini.
\newblock {A new subspace iteration method for the algebraic Riccati equation}.
\newblock {\em Numerical Linear Algebra with Applications}, 22(1):26--47, 2015.

\bibitem{MasOR14}
Arash Massoudi, Mark~R. Opmeer, and Timo Reis.
\newblock {Analysis of an iteration method for the algebraic {R}iccati
  equation}.
\newblock {\em SIAM Journal on Matrix Analysis and Applications},
  37(2):624--648, 2016.

\bibitem{MehT88}
V.~Mehrmann and E.~Tan.
\newblock {Defect correction methods for the solution of algebraic {R}iccati
  equations}.
\newblock {\em {{IEEE} Trans. Automat. Control}}, 33:695--698, 1988.

\bibitem{Ove92}
Michael~L. Overton.
\newblock {Large-Scale optimization of eigenvalues}.
\newblock {\em SIAM J. Optimiz.}, 2(1):88--120, 1992.

\bibitem{Pen00b}
T.~Penzl.
\newblock {{\textsc{Lyapack}} {U}sers {G}uide}.
\newblock Technical Report SFB393/00-33, Sonderforschungsbereich 393 {\it
  Numerische Simulation auf massiv parallelen Rechnern}, TU Chem\-nitz, 09107
  Chem\-nitz, Germany, 2000.
\newblock Available from \url{http://www.tu-chemnitz.de/sfb393/sfb00pr.html}.

\bibitem{Ruh94c}
Axel Ruhe.
\newblock {The rational Krylov algorithm for nonsymmetric eigenvalue problems.
  {III}: Complex shifts for real matrices}.
\newblock {\em {BIT}}, 34:165--176, 1994.

\bibitem{Saa09}
J.~Saak.
\newblock {\em {Efficient Numerical Solution of Large Scale Algebraic Matrix
  Equations in PDE Control and Model Order Reduction}}.
\newblock PhD thesis, TU Chemnitz, July 2009.
\newblock Available from
  \url{http://nbn-resolving.de/urn:nbn:de:bsz:ch1-200901642}.

\bibitem{Sab07}
J.~Sabino.
\newblock {\em {Solution of Large-Scale Lyapunov Equations via the Block
  Modified Smith Method}}.
\newblock PhD thesis, Rice University, Houston, Texas, June 2007.
\newblock Available from:
  \url{http://www.caam.rice.edu/tech\_reports/2006/TR06-08.pdf}.

\bibitem{ifiss}
D.~Silvester, H.~Elman, and A.~Ramage.
\newblock {{I}ncompressible {F}low and {I}terative {S}olver {S}oftware
  ({IFISS}) version 3.2}, May 2012.

\bibitem{Sim07}
V.~Simoncini.
\newblock {A new iterative method for solving large-scale Lyapunov matrix
  equations}.
\newblock {\em SIAM Journal on Scientific Computing}, 29(3):1268--1288, 2007.

\bibitem{Sim13}
V.~Simoncini.
\newblock {Computational methods for linear matrix equations}.
\newblock {\em {{SIAM} Rev.}}, 38(3):377--441, 2016.

\bibitem{SimSM13a}
V.~Simoncini, D.~Szyld, and M.~Monsalve.
\newblock {On two numerical methods for the solution of large-scale algebraic
  {R}iccati equations}.
\newblock {\em {{IMA} J. Numer. Anal.}}, 34(3):904--920, 2014.

\bibitem{Wac00}
E.~Wachspress.
\newblock {{ADI} iteration parameters for the {S}ylvester equation}, 2000.
\newblock Available from the author.

\bibitem{Wac13}
E.~Wachspress.
\newblock {\em {The {ADI} model problem}}.
\newblock Springer New York, 2013.

\bibitem{Wac88}
E.L. Wachspress.
\newblock {Iterative solution of the {L}yapunov matrix equation}.
\newblock {\em Appl. Math. Letters}, 107:87--90, 1988.

\bibitem{WonB05}
N.~Wong and V.~Balakrishnan.
\newblock {Quadratic alternating direction implicit iteration for the fast
  solution of algebraic {R}iccati equations}.
\newblock In {\em {Proc. Int. Symposium on Intelligent Signal Processing and
  Communication Systems}}, pages 373--376, 2005.

\bibitem{morWonB07}
Ngai Wong and V.~Balakrishnan.
\newblock {Fast positive-real balanced truncation via quadratic alternating
  direction implicit iteration}.
\newblock {\em {IEEE Trans. Comput.-Aided Design Integr. Circuits Syst.}},
  26(9):1725--1731, September 2007.

\end{thebibliography}

\end{document}